\documentclass[11pt]{amsart}
\usepackage{amssymb,amscd,amsmath,enumerate}
\usepackage[all]{xy}
\newtheorem{thm}{Theorem}
\newtheorem{theor}{Theorem}[section]

\newtheorem{corol}[theor]{Corollary}

\newtheorem{propo}[theor]{Proposition}

\newtheorem{defi}[theor]{Definition}

\newtheorem{lem}[theor]{Lemma}
\theoremstyle{remark}
\newtheorem{remark}[theor]{Remark}

\theoremstyle{definition}

\numberwithin{equation}{section}

\renewcommand{\bar}{\overline}

\newcommand{\Irr}{\mathrm{Irr}}

\newcommand{\R}{{\mathbb{R}}}

\newcommand{\F}{{\mathbb{F}}}
\newcommand{\N}{{\mathbb{N}}}
\newcommand{\Q}{{\mathbb{Q}}}
\newcommand{\Z}{{\mathbb{Z}}}

\newcommand{\uG}{{\underline{G}}}

\newcommand{\uX}{{\underline{X}}}
\newcommand{\uY}{{\underline{Y}}}
\newcommand{\uZ}{{\underline{Z}}}
\newcommand{\cG}{{\mathcal{G}}}
\newcommand{\cO}{{\mathcal{O}}}

\newcommand{\cY}{{\mathcal{Y}}}
\newcommand{\cZ}{{\mathcal{Z}}}

\newcommand{\tr}{\mathrm{tr}}

\newcommand{\sfw}{{\sf {W}}}
\newcommand{\sfx}{{\sf {X}}}
\newcommand{\sgn}{{\mathsf {sgn}}}

\newcommand{\Hom}{\mathrm{Hom}}

\renewcommand{\Pr}{\mathbf {P}}

\newcommand{\ZB}{\mathbf{Z}}
\newcommand{\CB}{\mathbf{C}}

\newcommand{\Gal}{\mathrm{Gal}}
\newcommand{\GL}{\mathrm{GL}}
\newcommand{\SL}{\mathrm{SL}}
\newcommand{\SU}{\mathrm{SU}}
\newcommand{\PSL}{\mathrm{PSL}}
\newcommand{\PGL}{\mathrm{PGL}}

\newcommand{\Or}{\mathrm{O}}

\newcommand{\Un}{\mathrm{U}}
\newcommand{\Sp}{\mathrm{Sp}}
\newcommand{\CSp}{\mathrm{CSp}}

\newcommand{\Span}{\mathrm{Span}}
\newcommand{\Stab}{\mathrm{Stab}}
\newcommand{\Fix}{{\sf fix}}

\newcommand{\Spec}{\mathrm{Spec}\;}
\newcommand{\Cl}{\mathrm{Cl}}

\newcommand{\ad}{\mathrm{ad}}
\renewcommand{\sc}{\mathrm{sc}}

\newcommand{\im}{\mathrm{im}}

\newcommand{\AAA}{{\sf A}}
\newcommand{\SSS}{{\sf S}}
\newcommand{\trans}{\mathrm{trans}}
\newcommand{\prim}{\mathrm{prim}}

\newcommand{\Alt}{{\raise 2pt\hbox{$\scriptstyle\bigwedge$}}}

\newcommand{\go}{\rightarrow}

\newcommand{\e}{\epsilon}

\begin{document}
\title{Probabilistic Waring problems for finite simple groups}

\author{Michael Larsen}
\email{mjlarsen@indiana.edu}
\address{Department of Mathematics\\
    Indiana University \\
    Bloomington, IN 47405\\
    U.S.A.}

\author{Aner Shalev}
\email{shalev@math.huji.ac.il}
\address{Einstein Institute of Mathematics\\
    Hebrew University \\
    Givat Ram, Jerusalem 91904\\
    Israel}

\author{Pham Huu Tiep}
\email{tiep@math.rutgers.edu}
\address{Department of Mathematics\\
    Rutgers University \\
    Piscataway, NJ 08854-8019 \\
    U.S.A.}

\begin{abstract}
The probabilistic Waring problem for finite simple groups asks whether
every word of the form $w_1w_2$, where $w_1$ and $w_2$ are non-trivial words
in disjoint sets of variables, induces almost uniform distributions on finite simple groups
with respect to the $L^1$ norm. Our first main result provides a positive solution to this problem.

We also provide a geometric characterization of words inducing almost uniform distributions
on finite simple groups of Lie type of bounded rank, and study related random walks.

Our second main result concerns the probabilistic $L^{\infty}$ Waring problem for finite
simple groups. We show that for every $l \ge 1$ there exists (an explicit) $N = N(l)=O(l^4)$, such that if
$w_1, \ldots , w_N$ are non-trivial words of length at most $l$ in pairwise disjoint sets of variables,
then their product $w_1 \cdots w_N$ is almost uniform on finite simple groups with respect to the
$L^{\infty}$ norm. The dependence of $N$ on $l$ is genuine.
This result implies that, for every word $w = w_1 \cdots w_N$ as above, the word map
induced by $w$ on a semisimple algebraic group over an arbitrary field is a flat morphism.

Applications to  representation varieties, subgroup growth, and random generation
are also presented. In particular we show that, for certain one-relator groups $\Gamma$,
a random homomorphism from $\Gamma$ to a finite simple group $G$ is surjective with probability
tending to $1$ as $|G| \go \infty$.
\end{abstract}

\subjclass{Primary 20P05; Secondary 11P05, 20C30, 20C33, 20D06, 20G40}

\thanks{ML was partially supported by NSF grants DMS-1401419 and DMS-1702152.
AS was partially supported by ISF grant 686/17 and the Vinik Chair
of mathematics which he holds. PT was partially supported by NSF grant DMS-1840702, and the Joshua Barlaz Chair in Mathematics.
The authors were also partially supported by BSF grant 2016072.}

\thanks{The paper is partially based upon work supported by the NSF under grant DMS-1440140 while AS and PT were in residence at MSRI (Berkeley, CA), during the Spring 2018
semester. It is a pleasure to thank the Institute for the hospitality and support.}

\thanks{The authors are grateful to the referee for careful reading and insightful comments that helped greatly improve the paper.}

\maketitle

\tableofcontents

\section{Introduction}

In the past two decades there has been much interest in word maps
and related Waring type problems (see for instance \cite{S2} and the references
therein).  Recall that a word is an element $w = w(x_1, \ldots , x_d)$ of the
free group $F_d$ on $x_1, \ldots , x_d$. Given any group $G$, the word $w$ gives rise to a
word map $w_G\colon G^d \to G$ induced by substitution.  When the group $G$ is understood,
we denote the map simply $w$.

Word maps on finite simple groups have attracted particular attention.
Here and throughout this paper, by a finite simple group we mean a non-abelian finite simple group.
Two words $w_1, w_2$ are said to be {\it disjoint} if they are words in disjoint sets of variables.
If $w = w_1w_2$ where $w_1, w_2 \ne 1$ are disjoint words, then it was shown in \cite{LST}
(following partial results from \cite{LS1, LS2})
that the word map $w$ is surjective on all sufficiently large finite simple groups.
This provides a best possible solution to the Waring problem for finite simple
groups, inspired by the classical Waring problem in number theory.

The \emph{probabilistic} Waring problem for finite simple groups asks whether, for $w=w_1w_2$
as above, the push-forward distribution $p_{w,G} = w_* U_{G^d}$  on a finite simple
group $G$ tends to the uniform distribution $U_G$
in the $L^1$ norm (see \eqref{def-norm}) as the order of $G$ tends to infinity.  That is, for a word $w$,
a finite group $G$ and an element $g \in G$, we let $p_{w,G}(\{g\})$ denote the probability
that $w(g_1, \ldots , g_d) = g$ when $g_i \in G$ are chosen uniformly and independently:
\[
p_{w,G}(\{g\}) = \frac{|w^{-1}(g)|}{|G|^d}.
\]
It is conjectured that, for finite simple groups $G$, we have
\[
\lim_{|G| \to \infty} \Vert p_{w,G}- U_G \Vert_{L^1}= 0
\]
(see for instance \cite[4.5]{S2}).
When this holds we say that $w$ is {\it almost uniform} on finite simple groups.

This conjecture has already been established in some cases.
In \cite[7.1]{GS} it is proved for $w = x_1^2 x_2^2$ (and it is also shown in \cite{GS} that the commutator
word $[x_1,x_2]$ is almost uniform). In \cite[1.1]{LS4} the conjecture is proved for $w = x_1^m x_2^n$
where $m, n$ are arbitrary non-zero integers.  It is also shown in \cite[1.2]{LS4} that admissible words, i.e.,
words in which each variable appears exactly twice, once as $x_i$ and once as $x_i^{-1}$, are almost uniform on
finite simple groups.
In \cite{LS1} the conjecture is established for arbitrary $w_1 w_2$ for alternating groups ${\sf A}_n$ (see Theorem 1.18 there and the
discussion following it). In this paper we prove the conjecture in full by confirming it for simple groups
of Lie type.

For a real function $f$ on a finite set $G$ and a real number $p > 0$, we define
\[
\Vert f \Vert_{L^p} = (|G|^{p-1} \sum_{g \in G} |f(g)|^p)^{1/p}.
\]
In particular,
\begin{equation}\label{def-norm}
\Vert f \Vert_{L^1} = \sum_{g \in G} |f(g)|,~
\Vert f \Vert_{L^{\infty}} = |G| \cdot \max_{g \in G} |f(g)|.
\end{equation}

Our first main result is as follows.

\begin{thm}
\label{prob-waring}
Let $w_1, w_2 \ne 1$ be disjoint words and let $w = w_1w_2$.  Then
\[
\lim_{|G| \to \infty} \Vert p_{w,G}- U_G \Vert_{L^1}= 0,
\]
where $G$ ranges over all finite simple groups.
\end{thm}

Let $G$ and $w \in F_d$ be as in Theorem \ref{prob-waring}. Then it follows from the theorem that, as $|G| \to \infty$
and $S \subseteq G^d$ satisfies $|S|/|G|^d \to 1$, we have
\[
|w(S)|/|G| \to 1,
\]
namely, almost all elements $g \in G$ can be expressed in the form $g = w(g_1, \ldots , g_d)$ where $(g_1, \ldots , g_d) \in S$.
Combining this observation with suitable known results we obtain some immediate applications of Theorem 1. For example,
using \cite[Theorem]{LiSh1} we deduce that {\it almost all elements of $G$ have the form $w(g_1, \ldots , g_d)$ where
$\langle g_i, g_j \rangle = G$ for all $1 \le i < j \le d$.}
The same holds if at least one of the following conditions holds: 
\begin{enumerate}[\rm(i)] 
\item all $g_i$ are regular semisimple if $G$ is of Lie type and bounded rank; 
\item
$n^{(1/2-\e)\log n} \le o(g_i) \le n^{(1/2+\e)\log n}$ for $G = \AAA_n$ where $\e > 0$ and $o(g)$ is the order of $g$
(see \cite[Theorem]{ET}); 
\item $|\CB_G(g_i)| \le q^{(1+\e)r}$ where $\e > 0$ and $G$ is classical of rank $r$ over $\F_q$.
\end{enumerate}
(This follows by combining Corollary 1.2(1) of \cite{FG} with Lemma 5.3 of \cite{S1}).

The proof of Theorem \ref{prob-waring} makes use of the classification of finite simple groups.  Since the result is asymptotic in nature, we do not need to consider sporadic groups at all, so it remains to deal with groups of Lie type.  For groups of classical type of unbounded rank, we combine arguments of combinatorial flavor with essential use of strong new character estimates proved in \cite{GLT2, GLT3}.  For groups of Lie type of bounded rank (including exceptional groups) we provide two proofs:
one character-theoretic, and the other geometric. The latter proof is based on the following characterization of almost uniform words
in bounded rank, which is of independent interest.

\begin{thm}\label{irred}
Let $r$ and $d$ be positive integers, and  $w\in F_d$ a non-trivial word.  The following conditions are equivalent:

\begin{enumerate}[\rm(i)]
\item As $G$ ranges over finite simple groups of Lie type of rank $\le r$,
$$\lim_{|G| \to \infty} \Vert p_{w,G}- U_G \Vert_{L^1}= 0.$$
\item For every prime  $p$ and every split simply connected semisimple group $\uG$ over
$\F_p$ of rank $\leq r$, there exists a power $q$ of $p$ such that
$$\lim_{n\to \infty} \frac{|w(\uG(\F_{q^n}))|}{|\uG(\F_{q^n})|}=1.$$
\item For every simply connected semisimple group $\uG$ of rank $\leq r$
over any field $\Bbbk$, the evaluation morphism $w\colon \uG^d\to \uG$ has geometrically irreducible generic fiber.

\end{enumerate}
\end{thm}

If these equivalent conditions hold  we say that $w$ is \emph{almost uniform in rank $\leq r$}.
If this is true for all $r$, we say that $w$ is \emph{almost uniform in bounded rank}.
Our geometric proof of Theorem \ref{prob-waring} in bounded rank is based on Theorem \ref{irred} above and
the fact that the generic fiber for any word of the form $w = w_1w_2$, where $w_1$ and $w_2$ are disjoint non-trivial words,
is geometrically irreducible.

Theorem \ref{irred} shows, in particular, that words that are surjective on large enough
finite simple groups of bounded rank are also almost uniform in bounded rank.
This is by no means obvious. In Segal's monograph \cite{Seg} a word $w \in F_d$ is said to be \emph{silly} if
$w \in x_1^{e_1} \ldots  x_d^{e_d} F_d'$ where $\gcd(e_1, \ldots , e_d) = 1$. It is observed in \cite[3.1.1]{Seg}
that silly words are precisely the words that are surjective on all groups.
It therefore follows that silly words are almost uniform in bounded rank.
In our next result we estimate the probability that a random word has the above properties.

For any $d>1$ and $n \ge 0$, we let $\sfw_{n,d}$ denote
the random element of $F_d$ obtained from an $n$ step random walk on $F_d$ with steps uniformly distributed in
$\{x_1^{\pm 1},\ldots,x_d^{\pm 1}\}$.

\begin{thm}
\label{positive}
\begin{enumerate}[\rm(i)]

\item
For all $d>1$ and all $n>0$, the probability that $\sfw_{n,d}$ is surjective on all groups exceeds $1/3$
and tends to $1$ as $d \to \infty$.

\item
For all $d>1$ and all $n>0$, the probability that $\sfw_{n,d}$ is almost uniform in bounded rank exceeds
$1/3$ and tends to $1$ as $d \to \infty$.
\end{enumerate}
\end{thm}

It has recently been shown in \cite[Theorem B]{CH} that a finite group $G$ is nilpotent if and only if
all words $w$ which are surjective on $G$ induce the uniform distribution $U_G$ on $G$. Using this and part (i)
of Theorem \ref{positive} it follows that {\it the probability that $\sfw_{n,d}$ is uniform on all finite
nilpotent groups exceeds $1/3$ and tends to $1$ as $d \to \infty$}.

Next, we turn to almost uniformity results with respect to other norms.

For the groups $\PSL_2(q)$, we can strengthen Theorem~\ref{prob-waring}, replacing the $L^1$ norm
by the $L^2$ norm. Indeed, by Corollary \ref{PSL2} below, if $G = \PSL_2(q)$ where $q$ ranges over prime powers,
then
\[
\lim_{q \to \infty} \Vert p_{w,G}- U_G \Vert_{L^2}= 0.
\]
It would be interesting to find out which families of finite simple groups satisfy this property.
We note that if $w$ is $[x_1,x_2]$ or $x_1^2x_2^2$ then we have
\[
\lim_{|G| \to \infty} \Vert p_{w,G}- U_G \Vert_{L^2}= 0
\]
as $G$ ranges over finite simple groups; indeed, this follows from \cite[\S2]{GS}.

It is also interesting to obtain almost uniformity results in the $L^{\infty}$ norm or the $L^p$ norm for arbitrary $p > 1$.
In this sense, for any fixed $k \ge 1$, the product
$w_1 \cdots  w_k$ of $k$ non-trivial pairwise disjoint words need not be almost uniform on all finite simple groups. Indeed we
may take $w_i = x_i^n$ for $n \geq k+2$. If $G$ is an alternating group of large degree (compared to $n$),
then it follows from Lemmas 2.17 and 2.18 of \cite{LiSh2} that $p_{x_i^n,G}(1) \ge |G|^{-1/n}$. If $G$ is a classical group of large rank
(compared to $n$) over a field with $q$ elements, then for some constant $c(n) > 0$ one has that
$$p_{x_i^n,G}(1) \ge q^{-c(n)}|G|^{-1/n}> |G|^{-1/(n-1)}$$
by \cite[Theorem 4.3]{LiSh3}.
Hence,
$$p_{w,G}(\{1\}) \ge |G|^{-k/(n-1)},$$
and $w$ is not almost uniform in $L^{\infty}$. If moreover
we take $n \geq kp/(p-1)+2$ then we see that $w$ is not almost uniform in $L^p$ whenever $p > 1$.

However, we do show in Corollary \ref{four} below, that if $w=w_1w_2w_3w_4$, a product of
four pairwise disjoint non-trivial words, then $w$ is almost uniform on
$\PSL_2(q)$ with respect to the $L^{\infty}$ norm. This result is best possible
in the sense that it fails to hold for some products of three disjoint words.
Indeed, it is shown in \cite[p. 1406]{S1} that $x_1^2x_2^2x_3^2$ is not almost uniform on
$\PSL_2(q)$ with respect to the $L^{\infty}$ norm.

Our second main result concerns the probabilistic Waring problem for finite simple groups
with respect to the $L^{\infty}$ norm.

\begin{thm}
\label{main-inf}
For a positive integer $l$ define $N(l) = 2\cdot 10^{18} l^4$. Let $N \ge N(l)$ be an integer and $w = w_1 \cdots  w_N$ a product of pairwise disjoint non-trivial words of length at most $l$. Then
\[
\lim_{|G| \to \infty} \Vert p_{w,G}- U_G \Vert_{L^\infty}= 0,
\]
where $G$ runs over all finite simple groups.
\end{thm}

This theorem generalizes Proposition 8.5 of \cite{S1} dealing with Lie type groups of bounded rank
(where $N$ depends on the rank $r$ and not on $l$; cf. Proposition \ref{waring-r}),
and Theorem 2.8 of \cite{S1} where $w_i$ are commutators and $N=2$.
Unlike Theorem \ref{prob-waring}, which was established long ago for alternating groups,
Theorem \ref{main-inf} is new (and highly non-trivial) also for ${\sf A}_n$ -- note that in this case the bound for $N(l)$ is
substantially smaller, see Proposition \ref{waring-alt}.
The proof of Theorem \ref{main-inf} is rather complicated, combining combinatorial and character methods.
In particular it follows from the theorem that
$x_1^l x_2^l \cdots  x_{N(l)}^l$ is almost uniform in $L^\infty$ on finite simple groups,
which may be regarded as a probabilistic non-commutative analogue of the Waring problem
in number theory. The discussion prior to Theorem \ref{main-inf} shows that the conclusion of
the theorem does not hold for $N \le l-2$, so the dependence of $N$ on $l$ is genuine.

Our third main result concerns flatness of certain word maps on algebraic groups,
representation varieties and subgroup growth of some one-relator groups,
as well as random generation of finite simple groups. While parts (i)--(iv) below
are applications of Theorem \ref{main-inf}, parts (v) and (vi) are more challenging
and require various additional tools.

Let us say that a word $w \in F_d$ is {\it even} if its image in the abelianization $F_d/[F_d,F_d]$
is a square, and that $w$ is {\it odd} otherwise (see also Definition \ref{parity1} below).

\begin{thm}
\label{applications}
For every positive integer $l$ there exists a positive integer $N^*(l)$ such that the following statement holds. Let $N \ge N^*(l)$ and $d$ be positive integers.
Suppose $w = w_1 \cdots  w_N\in F_d$ is a product of pairwise disjoint non-trivial words of length at most $l$,
and $\Gamma = \langle x_1, \ldots, x_d \; | \; w(x_1, \ldots , x_d) = 1 \rangle$. Then all the following statements hold.

\begin{enumerate}[\rm(i)]
\item For every field $\Bbbk$ and every semisimple algebraic group $\uG$ over $\Bbbk$,
the word morphism $w\colon \uG^d\to \uG$ is flat.

\item For every field $\Bbbk$ and every semisimple algebraic group $\uG$ over $\Bbbk$, the dimension of the $\Bbbk$-variety $\Hom(\Gamma, \uG)$, i.e., the Krull dimension of its coordinate ring, is
$(d-1)\dim \uG$.

\item For every field $\Bbbk$ and every positive integer $n$, the dimension
of the $\Bbbk$-variety $\Hom(\Gamma,\GL_n)$ is  $(d-1)n^2 + a$,
where $a=0$ if $w \not\in [F_d,F_d]$ and $a=1$ otherwise.

\item The number $a_n(\Gamma)$ of index $n$ subgroups of $\Gamma$ satisfies
\[
a_n(\Gamma) \sim bn \cdot (n!)^{d-2},
\]
where $b=1$ if $w$ is odd and $b=2$ if $w$ is even.
Thus
$\dfrac{a_n(\Gamma)}{a_n(F_{d-1})} \go b$ as $n \go \infty$.

\item The number $m_n(\Gamma)$ of maximal subgroups of $\Gamma$ of index $n$ satisfies
\[
m_n(\Gamma) \sim bn  \cdot(n!)^{d-2},
\]
where $b$ is as above. Thus $\dfrac{m_n(\Gamma)}{a_n(\Gamma)} \to 1$ as $n \to\infty$.

\item The probability that a random homomorphism from $\Gamma$ to a finite
simple group $G$ is an epimorphism tends to $1$ as $|G| \to \infty$.
\end{enumerate}
\end{thm}

Note that for statements (i)--(iv) of Theorem \ref{applications} to hold, it suffices to take $N^*(l) = N(l) = 2 \cdot 10^{18}l^4$ as
in Theorem \ref{main-inf}.

Some special cases of Theorem \ref{applications}, where $w_i$ are commutators or squares, were
already obtained in the past.

For example, in the case of surface words
$$w = x_1^{-1}x_2^{-1}x_1x_2\cdots  x_{2g-1}^{-1}x_{2g}^{-1}x_{2g-1}x_{2g}$$
for $g\ge 2$, part (i) of Theorem \ref{applications} was obtained in \cite[4.4]{AA}.  In characteristic zero it is also
shown in \cite[VIII]{AA} that, for $g\ge 374$, the fibers of $w_{\uG}$ have rational singularities.  It would be interesting to know whether the statement about rational singularities holds in the generality of part (i) of Theorem \ref{applications}, if $N$ is sufficiently large in terms of $l$.

Part (ii) of Theorem \ref{applications} for surface words (including non-oriented ones $w = x_1^2 \cdots  x_g^2$ where
$g \ge 3$) was obtained in \cite[1.11]{LiSh3}.

Part (iii) of Theorem \ref{applications} for (oriented and non-oriented) surface words was obtained in
\cite{RBK} and \cite{BK} for fields of characteristic zero (see also \cite{Go}), and in \cite[1.8]{LiSh3} for arbitrary fields.

Parts (iv) and (v) for surface groups were obtained in \cite{MP}.

For Fuchsian groups of genus $g \ge 2$ ($g \ge 3$ in the non-oriented case), a result similar to part (vi) of Theorem \ref{applications}
was obtained in Theorem 1.6 of \cite{LiSh3}.

We conclude the introduction with a result of independent interest, which plays an important
role in this paper and might be useful for other purposes.

\begin{thm}\label{epsilon} Let $w \in F_d$ be a non-trivial word, and let $G$ be a finite simple group.
Choose $g_1, \ldots, g_d \in G$ uniformly and independently. Then, for every $\e > 0$, the probability that
\[
|\chi(w(g_1, \ldots , g_d))| \le \chi(1)^{\e} \;\; {\rm for} \; {\rm all} \;\; \chi \in \Irr(G)
\]
tends to $1$ as $|G| \go \infty$.
\end{thm}

This result generalizes Proposition 4.2 of \cite{LS4} dealing with the case $w = x_1$,
and Theorem 7.4 of \cite{LS1} dealing with alternating groups.

The rest of the paper is organized as follows. In Section 2 we use methods from
algebraic geometry to prove Theorem \ref{irred} and deduce Theorem \ref{prob-waring}
for Lie type groups of bounded rank.
In Section 3 we discuss random walks and prove Theorem \ref{positive}.
In Section 4 we use character methods to provide an alternative proof of Theorem \ref{prob-waring}
in bounded rank (as well as some stronger results for $\PSL_2(q)$). In Section 5 we discuss classical groups
of large rank, and apply new character bounds obtained for them, and other tools, to complete
the proof of Theorem \ref{prob-waring}. Theorem \ref{epsilon} is also proved in Section 5, and
plays a key role in proving Theorem \ref{prob-waring}.  The proof of Theorem~\ref{main-inf}
is given in Section 6, and Section 7 is devoted to the proof of Theorem \ref{applications}.

\section{Geometric methods}

In this section we prove Theorem \ref{irred} and deduce Theorem \ref{prob-waring}
for Lie type groups of bounded rank.  At the end of the section, we prove a result which
will be needed below for Theorem~\ref{applications}.  Note that by an $\F_q$-\emph{variety},
we mean a separated, geometrically integral scheme of finite type over $\F_q$.

\begin{propo}
\label{23}
Let $\uX$ be an $\F_q$-variety, $\uY$ a disjoint union of $\F_q$-varieties $\uY_i$ of equal dimension,
and $f\colon \uY\to \uX$ a morphism defined over $\F_q$.
If $\uY$ is irreducible, then
\begin{equation}
\label{mostly-onto}
\lim_{n\to \infty} \frac{|f(\uY(\F_{q^n}))|}{|\uX(\F_{q^n})|} = 1
\end{equation}
if and only if $f$ is dominant and its generic fiber is geometrically irreducible.
In general,
\begin{equation}
\label{nearly-equi}
\lim_{n\to \infty} \Vert f_* U_{\uY(\F_{q^n})} - U_{\uX(\F_{q^n})}\Vert_{L^1} = 0,
\end{equation}
implies each restriction $f_i$ of $f$ to a component $\uY_i$ of $\uY$ is dominant and the generic fiber of each $f_i$ is geometrically irreducible.
\end{propo}

\begin{proof}

Let us first assume $\uY$ is irreducible.
By the Lang-Weil estimate, we may replace $\uX$, $\uY$, and $f$  by $\uX'$, $\uY'$, and $f' = f|_{\uY'}$ respectively, for any open subvariety $\uX'$ of $\uX$
and any open subvariety $\uY'$ of $f^{-1}(\uX')$.
Thus, we are justified in assuming $\uX$ and $\uY$ are affine and non-singular, and $f$ is dominant.  We denote their coordinate rings $A$ and $A_Y$ respectively.
As $\uX$ and $\uY$ are varieties, these are integral domains.
Let $K$ and $K_Y$ denote the fraction fields of $A$ and $A_Y$
respectively, and let $L$ denote the separable closure $L$ of $K$ in $K_Y$.  As $K_Y$ is a finitely generated field,
L is a finite extension of $K$.  Our claim is that $L=K$ if and only if (\ref{mostly-onto}) holds.

Choose $\alpha\in L\cap  A_Y$ to be a primitive element of $L/K$; after multiplying by a suitable element of $A$,
we may assume it is also integral over $A$.  Let $B=A[\alpha]\subset A_Y$, so $f$ factors through the finite morphism $\Spec B\to \Spec A$.
By \cite[Th\'eor\`eme~17.6.1]{EGA44}, $\Spec B\to \Spec A$ is \'etale in a neighborhood of the generic point of $\Spec B$, so
replacing $A$ by $A[1/a]$ for $\Spec A[1/a]$ small enough and $B$ and $A_Y$ by $B[1/a]$ and $A_Y[1/a]$ respectively,
we may assume $\Spec B\to \Spec A$ is
finite \'etale.  In particular $\Spec B$ is non-singular \cite[Th\'eor\`eme~17.11.1]{EGA44}.  Both $A$ and $B$ are therefore integrally closed, and $B$ is module-finite over $A$ and hence integral.
Thus $B$ is the integral closure of $A$ in $L$.

Let $M$ denote any finite extension of $L$ which is Galois over $K$ and $C$ the integral closure of $B$ in $M$.  Thus $C^{\Gal(M/K)}$ contains $A$ and has fraction field $K$.
It is contained in $K$ and integral over $A$, therefore equal to $A$.  Thus $\uX = \Spec A$ is the quotient of $\uZ=\Spec C$ by $\Gal(M/K)$.
Likewise, $B = C^{\Gal(M/L)}$, so $\uZ\to \uX$
factors through $\uY = \Spec B$.  Let $m$ denote the common dimension of $\uX$, $\uY$, and $\uZ$.
By the Lang-Weil estimate, $|\uX(\F_{q^n})|$, $|\uY(\F_{q^n})|$, and $|\uZ(\F_{q^n})|$ are all $(1+O(q^{-n/2}))q^{mn}$.

Applying the Chebotarev density theorem for $\uZ\to \uX$ \cite{Serre}, we see that in the limit $n\to \infty$, a positive proportion of
points in $\uX(\F_{q^n})$ split completely in $\uZ$ and therefore in $\uY$.
It follows that (\ref{mostly-onto}) implies $L=K$.

Conversely the condition $K=L$ is equivalent to the generic geometric irreducibility of $f$.
By \cite[Proposition~9.5.5,~Th\'eor\`eme~9.7.7]{EGA43}, we may assume without loss of generality
that all fibers of $f$ are geometrically irreducible and of equal dimension.
It is well known that the Lang-Weil theorem holds uniformly for families of varieties of the same dimension (see, e.g., \cite[Lemma 2.2]{LS2}),
and this implies (\ref{mostly-onto}) and even the stronger (\ref{nearly-equi}).

Finally, we consider the case that $\uY$ has irreducible components $\uY_1,\ldots,\uY_r$.
We note first that Lang-Weil implies that as $n\to \infty$, the probability of a random element of $\uY(\F_{q^n})$ lying in any fixed
$\uY_i(\F_{q^n})$ approaches $1/r$, so (\ref{nearly-equi}) implies that the restriction of $f$ to each $\uY_i$ is dominant.
Proceeding as before, we may assume that $\uX = \Spec A$ is affine, each $\uY_i$ is affine and geometrically connected over $\Spec B_i$,
$\Spec B_i$ is finite \'etale over $\uX$,  the fraction field $L_i$ of $B_i$ is a finite separable extension of the fraction field $K$ of $A$, and $M_i$
is a finite Galois extension of $K$ containing $L_i$.
Let $e = \dim \uY_i - \dim \uX$, the relative dimension of $\uY_i$ over $\uX$, which is the same for all $i$ since the $\uY_i$ have the same dimension and the morphisms to $\uX$ are all dominant.
By the uniform version of the Lang-Weil theorem,
for each $\F_{q^n}$-point of $\Spec B_i$, there are $(1+o(1))q^{ne}$ elements of $\uY_i(\F_{q^n})$ lying over it.

Applying the Chebotarev density theorem for $M_1\cdots M_r/K$, in the limit as $n\to \infty$, a positive proportion of points  $x\in \uX(\F_{q^n})$
split completely in each $L_i$, which means that there are $[L_i:K]$ $\F_{q^n}$-points of $\Spec B_i$ lying over $x$, therefore
$(1+o(1))[L_i:K]q^{ne}$ points of $\uY_i(\F_{q^n})$ lying over $x$, and, finally, $(1+o(1))([L_1:K]+\cdots+[L_r:K])q^{en}$
points of $\uY(\F_{q^n})$ lying over $x$.  If any of $L_1,\ldots,L_r$ is of degree $\ge 2$ over $K$, then this sum of degrees strictly exceeds $r$.
On the other hand, Lang-Weil implies
$$\lim_{n\to \infty} \frac{|\uY(\F_{q^n})|}{q^{en}|\uX(\F_{q^n})|} = 1.$$
Thus, $\Vert f_* U_{\uY(\F_{q^n})} - U_{\uX(\F_{q^n})}\Vert_{L^1}$ does not approach $0$.

\end{proof}

We now embark on the proof of Theorem \ref{irred}.

\begin{proof}
If $G$ is any finite group and $H$ is contained in its center, then for all $g\in G$,
$$[G:H]^d|w_{G/H}^{-1}(gH)| =  \sum_{h\in H} |w_G^{-1}(gh)|.$$
Defining $f\colon G^d\times H\to G$ by $f(g_1,\ldots,g_d,h) = w_G(g_1,\ldots,g_d)h$, we have
\begin{equation}
\label{translates}
\Vert f_* U_{G^d\times H} - U_G\Vert_{L^1} = \Vert p_{w,G/H} - U_{G/H}\Vert_{L^1}.
\end{equation}
On the other hand, the triangle inequality implies
\begin{equation}
\label{triangle}
\Vert p_{w,G/H}-U_{G/H} \Vert_{L^1} \leq \Vert p_{w,G}-U_G \Vert_{L^1}.
\end{equation}
We specialize to the case that $H$ is the center of $G$, while $G$ is
of the form $\uG(\bar\F_p)^F$, where $F$ is a generalized Frobenius map and $\uG$ is a  simply connected, split, simple algebraic
group of rank $\le r$ over $\F_p$.

To prove (i) implies (ii), given $p$ and $\uG$, we choose $q$ so that the center $Z$ of $\uG(\bar\F_p)$ is contained in $\uG(\F_q)$.
Applying Proposition~\ref{23} to the morphism $\uG^d\times Z\to \uG$ given by $(g_1,\ldots,g_d,z)\mapsto w(g_1,\ldots,g_d)z$,
condition (i) in the form given by (\ref{translates}) implies that each component of $\uG^d\times Z$ maps to $\uG$ with geometrically
irreducible generic fiber.  In particular this is true for the identity component, which is $\uG^d$.  A second application of Proposition~\ref{23}
gives (ii).

To prove (ii) implies (iii), we first note that generic geometric irreducibility is stable under base change of $\Bbbk$, so
we could assume without loss of generality that
$\Bbbk$ is algebraically closed and therefore that $\uG$ is split.
Since every split group is obtained by base change from a split group over a prime field,
we assume instead that $\uG$ is split and that $\Bbbk$ is either $\Q$ or $\F_p$ for some $p$.
Let $\cG$ denote a split semisimple group scheme over $\Z$ with
the same root system as $\uG$, and we denote  by $w_{\cG}$ the word morphism
$\cG^d\to \cG$ of schemes of finite type over $\Spec \Z$.
By \cite[Th\'eor\`eme~9.7.7]{EGA43}, the set of points of $\cG$ over which $w_{\cG}$
is geometrically irreducible is constructible and contains the generic point.  It therefore
contains a non-empty open set $S$.  By Chevalley's constructibility theorem \cite[Corollaire~1.8.5]{EGA41},
its image in $\Spec \Z$ is constructible and therefore contains all but finitely many closed points.
Thus $S$ contains
the generic point of all but finitely many fibers of $\cG\to \Spec\Z$, so
it suffices to prove the geometric irreducibility in the case that $k=\F_p$ and $\uG$ is split.
This case follows from Proposition~\ref{23}.

It remains to show that (iii) implies (i); by (\ref{triangle}), it suffices to prove
$$\lim_{|G|\to \infty} \Vert p_{w,G} - U_G \Vert_{L^1}= 0,$$
where $G$ ranges over groups
of the form $\uG_0(\bar\F_p)^F$, where $F$ is a generalized Frobenius map and $\uG_0$ is a  simply connected, split, simple algebraic
group of rank $\le r$ over $\F_p$.
We fix any root system $\Phi$ of rank $\leq r$ and prove the limit is zero as $G$ ranges over groups of this form with root system $\Phi$.
In the case that $F$ is a standard Frobenius map, $\uG_0(\bar\F_p)^F = \uG(\F_q)$
for some simply connected $\uG$ of rank $\le r$ and some $q$.  Thus, (i) follows from Proposition~\ref{23}.
In the case of Suzuki or Ree groups, it follows from the following lemma.
\end{proof}

\begin{lem}

Let $\uG$ be a split simple algebraic group over $\F_p$ and $f\colon \uG^d\to \uG$
a morphism of schemes.  There exists a constant $C$ such that if
$\bar x\in \uG(\bar\F_p)$ is a geometric point of $\uG$ such that $w^{-1}(\bar x)$ is irreducible of dimension $k$, and $F\colon
\uG_{\bar \F_p}\to \uG_{\bar \F_p}$ an endomorphism which preserves $w^{-1}(\bar x)$ and
such that $F^2$ is a standard $	p$-Frobenius endomorphism, and $s$ is a sufficiently large integer, then
$$\bigm|\bigm| w^{-1}(\bar x) (\bar \F_p)^{F^{2s+1}}\bigm| - p^{(2s+1)k/2}\bigm| \leq C p^{(2s+1)k/2-1/4}.$$
\end{lem}

\begin{proof}
We  fix $\ell\neq p$.
By the finiteness and proper base change theorems for \'etale cohomology over a field
we see that
for all $i$, $\dim H^i_c(w^{-1}(\bar x),\Q_\ell)$ is bounded as $\bar x$ varies.

We would like to apply the Lefschetz trace formula to count the $F$-fixed points
of $w^{-1}(\bar x)$.  We use Fujiwara's theorem (formerly Deligne's conjecture) \cite{Fu}.
If $F$ is an endomorphism of $\uG$ whose square is the $p$-Frobenius, then the
naive Lefschetz trace formula applies to all sufficiently high odd powers of $F$:
\begin{equation}
\label{Var}
|w^{-1}(\bar x) \cap (\uG(\bar \F_p)^{F^{2s+1}})^d|
= \sum_{i=0}^{2k} (-1)^i\tr(F^{2s+1}\vert H^i_c(w^{-1}(\bar x),\Q_\ell)).
\end{equation}
Since $F^2$ is a standard $p$-Frobenius map, by \cite[3.3.1]{D2} the eigenvalues
of $F$ on $H^i_c(w^{-1}(\bar x),\Q_\ell)$ have absolute value at most $p^{i/4}\leq \sqrt p^{\dim w^{-1}(\bar x)}$.  As $w^{-1}(\bar x)$ is a variety, its top cohomology group, $H^{2k}(w^{-1}(\bar X),\Q_\ell)$, is $1$-dimensional,
and $F^2$ acts with eigenvalue $p^k$.  Thus $F$ acts on the top cohomology with eigenvalue $\pm p^{k/2}$, and as left hand side of (\ref{Var}) is non-negative, for $f$ sufficiently large,
the eigenvalue is $p^{k/2}$, and the number of $F^{2s+1}$-fixed points differs from $p^{(2s+1)k/2}$ by $O(q^{(2s+1)k/2-1/4})$.
The lemma follows.
\end{proof}

An immediate consequence of Theorem~\ref{irred} is the following.

\begin{corol}
\label{abelian}
If the image of $w\in F_d$ in the abelianization $\Z^d = F_d/[F_d,F_d]$ is primitive, then $w$ is almost uniform in bounded rank.
\end{corol}

\begin{proof}
It suffices to prove that $w(G) = G$ for all groups $G$.  Indeed, as shown in \cite[3.1.1]{Seg}, if the image of $w$ in $\Z^d$ is a primitive $d$-tuple $(a_1,\ldots,a_d)$, we fix $b_1,\ldots,b_d\in\Z$ such that $\sum_i a_i b_i = 1$.
Then $w(g^{b_1},\ldots,g^{b_d}) = g$.
\end{proof}

We can now deduce Theorem~\ref{prob-waring} for Lie type groups of bounded rank. It follows
immediately from Theorem~\ref{irred} together with the following lemma.

\begin{lem}
Let $w=w_1 w_2 \in F_d$ where $w_1, w_2 \neq 1$ are disjoint words, and let $\uG$ be a semisimple simply connected algebraic group.
Then $w\colon \uG^d\to \uG$ has geometrically irreducible generic fiber.
\end{lem}

\begin{proof}
It suffices to prove this in the case that $\uG$ is simple.  In this case,
$G = \uG(\F_q)$ is the universal central extension of a finite simple group if $q$ is sufficiently large.  By Borel's theorem, $w_1$ and $w_2$ define dominant morphisms, so if $q$ is sufficiently large, there exist regular semisimple conjugacy classes $C_1$ and $C_2$ of $G$ lying in the image of $w_1$ and $w_2$ respectively.
By \cite[Lemma 5.1]{GT},
the image of $w$ contains all non-central semisimple elements of $G$ when $q$ is large, so condition (ii) of Theorem~\ref{irred} is satisfied.
Hence condition (iii) follows, as required.
\end{proof}

We conclude with a result which will be needed in \S7.
 \begin{propo}
 \label{flat}
 Let $w\in F_d$ be a word such that, as $G$ ranges over all finite simple groups of Lie type, we have
 \begin{equation}
 \label{uniform}
 \lim_{|G| \to \infty} \Vert p_{w,G}- U_G \Vert_{L^\infty}= 0.
 \end{equation}
 Then for every field $\Bbbk$ and every semisimple algebraic group $\uG$ over $\Bbbk$, the word map $w_{\uG}\colon \uG^d\to \uG$
 associated to $w$ is a flat morphism.
 \end{propo}

\begin{proof}
As flatness is not affected by faithfully flat base change \cite[Cor.~2.2.11 (iii)]{EGA42},  we can proceed as in Proposition~\ref{23}, observing that it suffices to consider the case that $\Bbbk$ is prime and $\uG$ is split.  Suppose we can prove flatness for
$\Bbbk=\F_p$ for all $p$ and therefore for $\Bbbk$ any finite field.
Let $\cG$ denote the split semisimple group scheme  over $\Spec \Z$ with the same root datum as $\uG$,
and let $w_{\cG}$ denote the word map $\cG^d\to \cG$.  Every non-empty closed set of $\cG^d$ contains a closed point.
By \cite[Th\'eor\`eme~11.1.1]{EGA43}, the flat locus of $w_{\cG}$ is open, so if it contains all closed points of $\cG^d$, it must be all of $\cG^d$.
Thus, we assume that $\Bbbk=\F_p$.

 By ``miracle flatness'' \cite[Proposition~6.1.5]{EGA42}, it suffices to prove that every fiber of $w_{\uG}$ has dimension $(d-1)\dim \uG$, the inequality
 $$\dim w_{\uG}^{-1}(g) \ge (d-1)\dim \uG$$
 being automatic \cite[(5.5.2.1)]{EGA42}.  If there exists a point on $\uG$ over which the inequality is strict, then by Chevalley's semicontinuity theorem \cite[Th\'eor\`eme~13.1.3]{EGA43},
there exists a closed point $x$ with this property.  If $\uG^{\ad}$ denotes the adjoint quotient of $\uG$, then the image of $x$ in $\uG^{\ad}$ has the same property
 for the word map $w_{\uG^{\ad}}$.  Thus, we may assume $\uG$ is adjoint, and since it is also split, it suffices to consider the case that it  is absolutely simple.
 The closed point $x$ corresponds to a $\Gal(\F_q/\F_p)$-orbit of points of $\uG(\F_q)$ for some $q$, and we let $x_0$ denote a point in this orbit.  Thus,
 the fiber $F_{x_0}$ of $w_{\uG_{\F_q}}$ over $x_0$ has dimension at least $(d-1)\dim \uG+1$.

Replacing $\F_q$ by a finite extension field if necessary, we may assume that the fiber $F_{x_0}\subset \uG_{\F_q}^d$ has the property that
all of its irreducible components are geometrically irreducible.  We may further assume the same for the inverse image $F_{x_0}^{\sc}$ of $F_{x_0}$ in $(\uG^{\sc}_{\F_q})^d$:
$$\xymatrix{F_{x_0}^{\sc}\ar[r]\ar[d] & (\uG^{\sc})^d\ar[d] \\
F_{x_0}\ar[r]\ar[d] &\uG^d\ar[d]^{w_{\uG}} \\
\Spec \F_q\ar[r]^-{x_0}&\uG}$$

Let $G$ denote the derived group of $\uG(\F_q)$.  The image of $F_{x_0}^{\sc}(\F_q)$ in $F_{x_0}(\F_q)\subset \uG(\F_q)^d$ lies in
$$\im (\uG^{\sc}(\F_q)^d\to \uG(\F_q)^d)=G^d,$$
and by the Lang-Weil estimate, if $q$ is sufficiently large,
$$|F_{x_0}^{\sc}(\F_q)| > \frac{q^{1+(d-1)\dim \uG}}2,$$
so
$$|\im (F_{x_0}^{\sc}(\F_q)\to F_{x_0}(\F_q))| \ge \frac{q^{1+(d-1)\dim \uG}}{2|\ZB(\uG^{\sc}(\F_q))|^d}.$$
The denominator does not depend on $q$, so when $q$ is sufficiently large, this is inconsistent with (\ref{uniform}).
\end{proof}

%
%

\section{Random words}

This section is devoted to the proof of Theorem \ref{positive}.

For $n \ge 0$ and $d \ge 1$ let $\sfx_{n,d}$ denote the random variable associated with the standard random walk with $n$ steps in $\Z^d$.
We also set $\sfx_n = \sfx_{n,2}$.
Thus $\sfx_n$ is the probability distribution in $\Z^2$ corresponding to a random walk of length $n$ in which each step in the set $\{(\pm 1,0),(0,\pm1)\}$ has probability $1/4$.

\begin{lem}
\label{prec-ineq}
Let $(a,b), (a',b')\in \N^2$, with $a+b\equiv a'+b'\pmod 2$.  Then
$$\Pr[\sfx_n=(a,b)]\le \Pr[\sfx_n=(a',b')]$$
for all $n\in \N$
if either of the following conditions holds:
\begin{enumerate}
\item[{\rm (\ref{prec-ineq}.1)}] $a-a', b-b'\in \N$.
\item[{\rm (\ref{prec-ineq}.2)}] $a+b = a'+b'$, and $|a-b| \ge |a'-b'|$
\end{enumerate}
\end{lem}

\begin{proof}
We proceed by induction on  $n$, the claim being  trivial for $n=0$.
For any $a,b\in\Z$, we abbreviate  $\Pr[\sfx_n=(a,b)]$ by $(a,b)_n$.
Thus,
\begin{align*}
(a,b)_{n+1}&= \frac{(a-1,b)_n+(a+1,b)_n  +(a,b-1)_n + (a,b+1)_n}{4} \\
&=\frac{(|a-1|,b)_n+(a+1,b)_n  +(a,|b-1|)_n + (a,b+1)_n}{4}.
\end{align*}
We write $(a',b')\preceq_* (a,b)$ if $a,b,a',b'\in\N$ and  (\ref{prec-ineq}.$\ast$) holds for $\ast\in \{1,2\}$.
If $(a',b')\preceq_1 (a,b)$, then
$$(a'+1,b')\preceq_1 (a+1,b),~
(a',b'+1)\preceq_1 (a,b+1).$$
Moreover,
$$(|a'-1|,b')\preceq_1 (|a-1|,b)$$
unless $a'=0$ and $a=1$.  In this case, the parity condition implies $b\ge b'+1$, so

$$(|a'-1|,b') = (1,b')\preceq_2 (0,b'+1)\preceq_1(0,b) = (|a-1|,b).$$
Likewise,
$$(a',|b'-1|)\preceq_1 (a,|b-1|)$$
unless $b'=0$ and $b=1$, in which case
$$(a',|b'-1|) = (a',1)\preceq_2 (a'+1,0)\preceq_1 (a,0) = (a,|b-1|),$$
so (\ref{prec-ineq}.1) follows by induction.

Suppose, on the other hand, that $(a',b')\preceq_2 (a,b)$.
It suffices to consider the case that $a>a'\ge b' > b$, so
$a$ and $a'$ are positive.  It follows that
\begin{align*}
(|a'-1|,b')= (a'-1,b')&\preceq_2 (a-1,b) = (|a-1|,b),\\
(a'+1,b')&\preceq_2 (a+1,b),\\
 (a',b'+1)&\preceq_2 (a,b+1).
\end{align*}
If $b>0$, then
$$(a',|b'-1|) = (a',b'-1)\preceq_2 (a,b-1) = (a,|b-1|),$$
and we are done by induction.  If $b=0$, then
$$(a',|b'-1|) = (a',b'-1) \preceq_2(a,1) = (a,|b-1|),$$
and we are done by induction.
\end{proof}

\begin{propo}
\label{positive-n}
If $p>2$ is prime and $n>0$, then
$$\Pr[\sfx_n\in p\Z^2\setminus\{0,0\}] < \frac {4}{(p+1)^2}\ .$$
\end{propo}

\begin{proof}
Let $\cZ = \Z_{>0}\times \N$.
If $R$ is the group of automorphisms of $\Z^2$ generated by rotation by $\pi/2$, then
$\Z^2\setminus\{0,0\}$ is the disjoint union of $\rho(\cZ)$ for all $\rho\in R$.
Let $\cZ_p = p\Z^2\cap \cZ$.
As $|R|=4$, it suffices to prove that
\begin{equation}
\label{quarter}
\Pr[\sfx_n\in \cZ_p] < \frac 1{(p+1)^2}.
\end{equation}
For $(a,b)\in \cZ_p$, we define subsets $\cY(a,b)$ of $\cZ$ as follows.  For $b=0$,
$$\cY(a,0) = \cZ\cap \bigcup_{\rho\in R} \rho\{(x,y)\in \Z^2: |x-a| + |y| \in 2\Z\cap [0,p],\,x+|y|\le a\},$$
and for $b > 0$,
$$\cY(a,b) = \{(x,y)\in \cZ: |x-a|+|y-b| \in 2\Z\cap [0,p],\,|x|\le a,\,|y|\le b\}.$$
By Lemma~\ref{prec-ineq}, $(x,y)\in \cY(a,b)$ implies $(x,y)_n\ge (a,b)_n$ for all $n$.

We claim the sets $\{\cY(a,b)\mid (a,b)\in \cZ_p\}$ are pairwise disjoint.  Indeed, for $\cY(a_1,b_1)\cap \cY(a_2,b_2)$ to be
non-empty for distinct elements $(a_1,b_1)$ and  $(a_2,b_2)$ of $\cZ_p$, it is necessary that
$a_1+b_1\equiv a_2+b_2\pmod 2$, and this together with the fact that $a_1,b_1,a_2,b_2\in p\Z$ implies that the
 the $L^1$
distance between any point in the $R$-orbit of $(a_1,b_1)$ and any point in the $R$-orbit of
$(a_2,b_2)$ is at least $2p$.
On the other hand, all the elements of $\cY(a,b)$ are within distance $p-1$
in the $L^1$ norm of
some element of the $R$-orbit of $(a,b)$.

Whether $b$ is $0$ or not, $|\cY(a,b)| = (p+1)^2/4$. Thus,
$$(a,b)_n \le \frac 4{(p+1)^2} \Pr[\sfx_n\in \cY(a,b)].$$
By symmetry, $\Pr[\sfx_n\in \cZ] \le 1/4$, so
\begin{align*}
\Pr[\sfx_n\in \cZ_p] = \sum_{(a,b)\in \cZ_p} (a,b)_n &\le \frac 4{(p+1)^2} \Pr[\sfx_n \in \bigcup_{(a,b)\in \cZ_p} \cY(a,b)] \\
									&< \frac 4{(p+1)^2}\Pr[\sfx_n \in \cZ] = \frac{\Pr[\sfx_n\neq (0,0)]}{(p+1)^2}, \\
									\end{align*}
implying (\ref{quarter}).
\end{proof}


\begin{proof}[Proof of Theorem \ref{positive}]
Part (ii) of the theorem follows from part (i) and Theorem \ref{irred},
so it suffices to prove the two assertions in part (i).

Let $\sfw_{n,d} = w_1 \cdots  w_n$, where the $w_i$ are chosen independently from the standard generating set $\{x_1^{\pm 1},\ldots,x_d^{\pm 1}\}$, with all elements equally likely, and $n>0$.  Let $\phi\colon F_d \to \Z^d$ be the abelianization map.  Thus $\phi(\sfw_{n,d})$ is exactly $\sfx_{n,d}$.

We first assume $d=2$, so $\sfx_{n,d}$ is just $\sfx_n$, and
the probability that $\phi(\sfw_{n,d})$ is primitive is the probability $P_n$
that an $n$ step random walk in $\Z^2$ gives a primitive element.

By \cite[3.1.1]{Seg}, if $\phi(w)$ is primitive, then $w$ is surjective on all groups, and by Theorem~\ref{irred}, it is almost uniform
in bounded rank.
Thus, to prove the  theorem for $d=2$, it suffices to prove that $P_n>1/3$ for all $n>0$.
Now, if $a_{n,m}$ denotes the probability that $\sfx_n\neq (0,0)$ and the g.c.d. of $m$ and the two coordinates of $\sfx_n$ is $>1$, then
$$P_n \ge  1 -  a_{n,6} - \sum_p a_{n,p}-(0,0)_n,$$
where $p$ ranges over primes $\ge 5$, so
$$\inf_{n\ge 1} P_n \ge 1 - \sup_n a_{n,6} - \sum_p \sup_n a_{n,p} - (0,0)_n.$$

To estimate $\sup_n a_{n,6}$, we fix a cutoff $N$ and calculate $a_{n,6}$ for $n\le N$ (using interval arithmetic to get a rigorous upper bound).
To bound $a_{n,6}$ for $n\ge N$, we consider the $n$ step random walk on $(\Z/6\Z)^2$ in which the steps $(\pm 1,0)$, $(0,\pm 1)$
each have probability $1/4$.  An upper bound for the probability of any state occurring  in $n\ge N$ steps is given by the maximum over all states of
the probability of occurrence in $N$ steps.  Since the image of $(2\Z)^2\cup (3\Z)^2$
in $(\Z/6\Z)^2$ has $12$ elements
of which $10$ have even coordinate sum and $2$ have odd coordinate sum,
the probability of
landing in $(2\Z)^2\cup (3\Z)^2$ after $n\ge N$ steps is at most $10$ times the maximum probability at time $N$
of any state  in the (mod $6$) Markov chain.

Likewise, for any given $p\ge 5$, to estimate $\sup_n a_{n,p}$, we can fix a cutoff $N$ and proceed as before.
In practice, to obtain a good bound, $N$ should  be chosen of order $p^2$.  We use this method for small $p$,
while for large $p$, we use the estimate $\sup_n a_{n,p} < 4/(p+1)^2$ given by Proposition~\ref{positive-n}.
For $n\ge N$, $(0,0)_n$ is bounded above by the maximum of $(a,b)_N$ over pairs $(a,b)\in \N$.
Implementing these calculations by computer using $N=1000$,
$$(0,0)_n \le .0006,\;a_{n,6} < .5556,\; a_{n,5} < .0401,\;a_{n,7} < .0205,\;\ldots,\;a_{n,59} < .0007$$
for all $n\ge 1000$, so
\begin{align*}
\inf_{n\ge 1000} P_n &> 1-.5556-.0401-.0205-.0083-\cdots - .0007 \\
                                               &\qquad-\sum_{p>60} \frac 4{(p-1)^2} - .0006 \\
                                               &> .3535 - \sum_{60<p<10003} \frac 4{(p-1)^2} - \int_{10000}^\infty \frac {2\,dx}{x^2} \\
                                               &= .3535 - .0132-.0005 - .0006 > \frac 13 + \frac{1}{3^5}.
\end{align*}
This proves the the theorem for $n\ge 1000$; and for $1\le n < 1000$, machine computation shows that the probability that the coordinates of $\sfx_n$ are relatively prime is greater than $.4$.

We now consider the general case $d\ge 2$.
Recall that $\sfx_{n,d}$ denotes the random variable associated with the standard random walk with $n$ steps in $\Z^d$.
It suffices to prove that the probability that $\sfx_{n,d}$ is primitive always exceeds $1/3$ and tends to $1$ as
$d \to \infty$.  

%

Let $\pi\colon \Z^d\to \Z^2$ denote the projection map onto the first two coordinates.
For $\sfx_{n,d}$ to be primitive, it suffices that $\pi(\sfx_{n,d})$ is primitive.  If we condition on there being $n'$ terms in $w_1,\ldots,w_n$ which lie in $\{x_1^{\pm1},x_2^{\pm1}\}$, then $\pi(\sfx_{n,d})$ has the same probability distribution as $\sfx_{n',2}$, so the probability of primitivity is at least $1/3$, unless $n'=0$.  The conditional probability that $\sfx_{n,d}$ is primitive, assuming $n'=0$ is the same as the probability that $\sfx_{n,d-2}$ is primitive.   We can therefore proceed by induction, but we must check the additional base case $d=3$.  Here, the probability that $n'=0$ is $3^{-n}$, so if $n\ge 5$, the probability that $\sfx_{n,3}$ is primitive is greater than $1/3$.  The probabilities of imprimitivity for $\sfx_{1,3}$, $\sfx_{2,3}$, $\sfx_{3,3}$, and $\sfx_{4,3}$  are $1$, $\frac23$, 
$\frac{35}{36}$, and $\frac{20}{27}$, respectively.

If $n$ is fixed and $d\to \infty$, the probability approaches $1$ that $\phi(w_1),\ldots,\phi(w_n)$ are linearly independent, which implies that $\sfx_{n,d}$ is primitive.  On the other hand, for any $k>0$, as $n$ and $d$  both grow without bound
$$\Pr[\Span(\phi(w_1),\ldots,\phi(w_n))\ge k]$$
goes to $1$.  Assuming the span has dimension $\ge k$ and $d\ge 2k$, there exist $k$ disjoint pairs of coordinates
such that each projection of the random walk associated to one of the $k$ pairs $(i,j)$ satisfies $n_{i,j}>0$, and therefore, conditioning on the choice of the $k$ pairs, the probability that each of the $k$  projections of $\sfx_{n,d}$ is imprimitive is less than $(2/3)^k$.
Thus, $\sfx_{n,d}$ is primitive with probability greater than $1-(2/3)^k$. Taking $k\to\infty$, this implies the second assertion in part (i) of the theorem and completes the proof.
\end{proof}

\begin{remark}
For any odd number $m$, the Markov chain on $(\Z/m\Z)^2$ given by our $\pmod m$ random walk is
irreducible and aperiodic, since the set of possible steps does not lie in a single coset of any proper subgroup of $(\Z/m\Z)^2$.
Therefore, it converges to the unique invariant distribution, which is the uniform distribution.  It follows that
$$\lim_{n\to \infty} a_{n,m} = 1-\prod_{p\vert m} (1-p^{-2}).$$
For $m$ even, the situation is slightly more complicated, since for $n$ odd, $a_{n,2}=0$ and for $n>0$ even, $a_{n,2}=1/2$.  Thus,
$$\lim_{n\to \infty} a_{2n,m} = 1-(2/3)\prod_{p\vert m} (1-p^{-2})$$
while
$$\lim_{n\to \infty} a_{2n+1,m} = 1-(4/3)\prod_{p\vert m} (1-p^{-2}).$$
From this together with Proposition~\ref{positive-n} it is easy to deduce that
$$\limsup_{n\to \infty} P_n = \frac 4{\pi^2},$$
and it follows, without any necessity for computer calculation, that there exists a  positive lower bound for $P_n$ for all $n>0$.  We do not know whether $P_n>4/\pi^2$ for all $n>0$.
\end{remark}

%

%

\section{Character methods}

In this section we provide an alternative proof of Theorem \ref{prob-waring} for Lie type groups of bounded rank using character theory.
We also prove a stronger $L^2$ result in the case $G = \PSL_2(q)$
by studying the non-commutative Fourier expansion of the probability distribution $p_{w,G}$.

\begin{lem}
\label{ss}
Let $w \in F_d$ be a non-trivial word.
Let $\uG(\F_q)$ be a finite simple group of Lie type of rank $r$ over a field with $q$ elements.
Let $S$ be the set of regular semisimple elements of $\uG(\F_q)$. Then we have
\[
p_{w,\uG(\F_q)}(S) \ge 1 - cq^{-1},
\]
where $c > 0$ depends on $w$ and $r$ but not on $q$.
\end{lem}

\begin{proof}

At the level of the algebraic group $\uG$, the regular semisimple elements form an open
dense subset, and its complement is a proper subvariety. By Borel's theorem \cite{Borel}
the inverse image of this subvariety under the word map induced by $w$ on $\uG^d$ is a proper
subvariety of $\uG^d$. By the Lang-Weil estimate,
\[
p_{w,\uG(\F_q)}(\uG(\F_q) \setminus S) \leq cq^{-1},
\]
yielding the desired conclusion.
\end{proof}

Next, let $w_1, w_2$ be non-trivial disjoint words, and let $G = \uG(\F_q)$ be
as above. Let $C_1, C_2$ be conjugacy classes of regular semisimple elements of $G$,
and let $g$ be a regular semisimple element of $G$. For $i = 1,2$ choose $x_i \in C_i$ uniformly
and independently. It is well known that the probability $p(C_1,C_2,g)$ that $x_1 x_2 = g$ satisfies
\begin{equation}
\label{Frobenius}
p(C_1,C_2,g) = |G|^{-1} \sum_{\chi \in \Irr(G)} \frac{\chi(C_1)\chi(C_2)\chi(g^{-1})}{\chi(1)}.
\end{equation}
It is known that there exists a constant $b$ depending only on $r$ such that
$|\chi(s)| \leq b$ for all regular semisimple elements $s \in G$ (see for instance
\cite[4.4]{S1}). This yields
\[
|p(C_1,C_2,g)-|G|^{-1}| \leq |G|^{-1} \sum_{1_G \ne \chi \in \Irr(G)} b^3/\chi(1) = b^3|G|^{-1} (\zeta_G(1)-1),
\]
where $\zeta_G(s) = \sum_{\chi \in \Irr(G)} \chi(1)^{-s}$ is the Witten zeta function of $G$.
Suppose $G \ne \PSL_2(q)$. Then we have $\zeta_G(1) \to 1$ as $|G| \to \infty$ by \cite[1.1]{LiSh3}.
This yields
\[
p(C_1,C_2,g) = |G|^{-1}(1+o(1)),
\]
for all $C_1,C_2,g$ as above. Summing up over $C_1, C_2$ and applying Lemma \ref{ss}
we see that for every $\e > 0$ and large enough $G$, for at least $(1-\e)|G|$ elements
$g \in G$ we have $p_{w_1w_2,G}(\{g\}) \ge (1-\e)|G|^{-1}$. This easily yields
\[
\Vert p_{w_1w_2,G} - U_G\Vert_{L^1} \to 0
\]
as $|G| \to \infty$.
This proves Theorem \ref{prob-waring} for bounded rank Lie-type groups $G \ne \PSL_2(q)$.

In the case $G = \PSL_2(q)$ we obtain a somewhat stronger result, see Corollary \ref{PSL2} below.
We need some preparations.

Let $G$ be a finite group, $w \in F_d$ a word, and $p_{w,G}$ its induced probability distribution on $G$.
We express the corresponding class function $P_{w,G}$ as a linear combination of irreducible characters
\[
P_{w,G} = |G|^{-1}\sum_{\chi \in \Irr(G)} a_{w,\chi} \chi.
\]
It is well known (see for instance \cite[\S8]{S1}) that if  $w_1, w_2$ are disjoint words, then we have
\[
a_{w_1w_2, \chi} = a_{w_1, \chi} a_{w_2, \chi}/ \chi(1)
\]
for all $\chi \in \Irr(G)$.
Using an inverse Fourier transform one obtains
\[
a_{w, \chi} = |G|^{-d} \sum_{g_1, \ldots , g_d \in G} \chi(w(g_1, \ldots , g_d)^{-1})
= \sum_{g \in G} P_{w,G}(g) \chi(g^{-1}).
\]

The following result, which may be of some independent interest, will be
useful in this section.

\begin{propo}\label{bounded}  For every word $1 \ne w \in F_d$ there exists a positive number $c(w)$ such that
for every group $G = \PSL_2(q)$ and every character $\chi \in \Irr(G)$ we have $|a_{w,\chi}| \leq c(w)$.
\end{propo}

\begin{proof} Inspecting the well known character table of $G$, we see that, if
$1 \ne g \in G$ and $\chi \in \Irr(G)$, then $|\chi(g)| \leq 2$ except if $g$ is unipotent.
In this case we have $|\chi(g)|>2$ for at most two irreducible characters $\chi$, and in any case,
$|\chi(g)| \leq q^{1/2}$.
Let $S \subset G$ be the set of (regular) semisimple elements and let $U$ be the set of (regular) unipotent
elements.
Then, at the level of algebraic groups, $U$ is contained in a proper subvariety, and it follows
from Borel's theorem \cite{Borel} and the Lang-Weil theorem
that $p_{w,G}(U) \leq e q^{-1}$ for some constant $e = e(w)$.

We have
\[
|a_{w, \chi}|  \leq \sum_{g \in G} P_{w,G}(g) |\chi(g)|
\leq 2p_{w,G}(S) + p_{w,G}(U)q^{1/2} + P_{w,G}(1)\chi(1).
\]
Since $P_{w,G}(1) \leq f(w) q^{-1}$ and $\chi(1) \leq q+1$ this yields
\[
|a_{w, \chi}| \leq 2 + e(w)q^{-1}q^{1/2} + f(w)q^{-1}(q+1) \leq 2 + e(w)q^{-1/2} + 2f(w) \leq c(w)
\]
for a suitable $c(w)$.
\end{proof}

The above result has some applications, as follows.

\begin{corol}\label{PSL2}  Let $w = w_1w_2$ where $w_1, w_2 \in F_d$ are non-trivial disjoint words.
If $G = \PSL_2(q)$ where $q$ ranges over prime powers, then we have
\[
\lim_{q \to \infty} \Vert p_{w,G}- U_G \Vert_{L^2}= 0.
\]
\end{corol}

\begin{proof}
This follows easily using non-commutative Fourier methods. For $\chi \in \Irr(G)$ we have
\[
|a_{w,\chi}| = \frac{|a_{w_1, \chi}| |a_{w_2, \chi}|}{\chi(1)} \leq \frac{c(w_1)c(w_2)}{\chi(1)}
\]
by Proposition \ref{bounded}.
Applying \cite[Lemma 2.2]{GS} we obtain
\begin{equation}
\label{L2}
(\Vert p_{w,G} - U_G \Vert_{L^2})^2 \leq \sum_{1_G \ne \chi \in \Irr(G)} |a_{w, \chi}|^2 \leq c(w_1)^2c(w_2)^2(\zeta_G(2)-1),
\end{equation}
where $\zeta_G$ is as before.
By \cite[Theorem 1.1]{LiSh3} the RHS of (\ref{L2}) tends to $0$ as $|G| \to \infty$ for all finite
simple groups $G$.
This completes the proof.
\end{proof}

Note that the above result completes the proof of Theorem \ref{prob-waring} for Lie-type groups
of bounded rank. We also obtain an $L^{\infty}$ result as follows.

\begin{corol}\label{four}
Let $w_1, w_2, w_3, w_4$ be pairwise disjoint non-trivial words. Let $w = w_1 w_2 w_3 w_4$ and $G = \PSL_2(q)$.
Then
\[
\lim_{q \to \infty} \Vert p_{w,G}- U_G \Vert_{L^{\infty}}= 0.
\]
\end{corol}

\begin{proof} Note that
\[
a_{w,\chi} = \frac{a_{w_1,\chi}a_{w_2,\chi}a_{w_3,\chi}a_{w_4,\chi}}{\chi(1)^3}.
\]
Combining this with Proposition \ref{bounded} we obtain
\[
|a_{w,G}| \leq \frac C{\chi(1)^3},
\]
where $C = c(w_1)c(w_2)c(w_3)c(w_4)$.

Proposition 8.1 of \cite{S1} shows that
\[
\Vert p_{w,G} - U_G \Vert_{L^{\infty}} \leq \sum_{1_G \ne \chi \in \Irr(G)} |a_{w,\chi}| \chi(1).
\]
This yields
\[
\Vert p_{w,G} - U_G \Vert_{L^{\infty}} \leq \sum_{1_G \ne \chi \in \Irr(G)} C \chi(1)^{-2} = C(\zeta_G(2)-1).
\]
As noted above, the right hand side tends to $0$ as $|G| \to \infty$, completing the proof.
\end{proof}

A similar statement for three words is false. Indeed, it is shown in \cite[p. 1406]{S1}
that for $w = x_1^2 x_2^2 x_3^2$ and $G = \PSL_2(q)$, $p_{w,G}$ is not almost uniform in
$L^{\infty}$.

Finally, it is easy to see that the bound on the Fourier coefficients in Proposition \ref{bounded}
cannot hold for all finite simple groups; indeed words of the kind $w = x_1^n$ give counter-examples.
However, we conjecture that, for every non-trivial word $w$ there exist a real number $\epsilon(w) > 0$
and a positive integer $N(w)$ such that, for all finite simple groups $G$ of order at least $N(w)$ and for all
characters $\chi \in \Irr(G)$ we have
\[
|a_{w,\chi}| \leq \chi(1)^{1-\epsilon(w)}.
\]

\section{The $L^1$ Waring problem}

In this section we prove Theorem \ref{epsilon} and use it to complete the proof of Theorem~\ref{prob-waring}.

Recall that Theorem \ref{epsilon} is known for alternating groups, see \cite[Theorem 7.4]{LS1}.
For groups of Lie type $G = \uG(\F_q)$ of bounded rank, it follows easily from Lemma \ref{ss} above. (Indeed,
when $|\uG(\F_q)| \to \infty$, $q$ tends to infinity, and Lemma \ref{ss} then shows that the probability that
$w(g_1, \ldots,g_d)$ is regular semisimple tends to $1$.
The character values of regular semisimple elements are bounded in term of the rank $r$ of $G$, whereas
$\chi(1)$ is at least of the magnitude of $q^r$, whence Theorem \ref{epsilon} follows.)
Hence it remains to prove Theorem \ref{epsilon} for simple classical groups
of arbitrarily high rank (which, in particular, can be assumed to be of type $A_r$, $^2A_r$, $B_r$, $C_r$, $D_r$, or $^2D_r$.)
Thus we may (and do) assume that $G$ is a simple classical group of Lie type whose rank can be taken as large as we wish.

Let $H$ be a group satisfying for some $n$ and $q$ one of the following conditions:
$\SL_n(q) \lhd H \leq \GL_n(q)$, $\SU_n(q) \lhd H \leq \Un_n(q)$, $\Sp_n(q)\lhd H\leq \CSp_n(q)$ (with $2|n$),
or $\Omega^{\pm}_n(q)\lhd H\leq \Or^{\pm}_n(q)$.  We will consider the natural action of
$H$ on $V = \F_q^n$, $\F_{q^2}^n$, $\F_q^n$, $\F_q^n$, which in the last three cases is endowed with a non-degenerate
$H$-invariant Hermitian, symplectic, or quadratic form $\langle\,,\,\rangle$.   In the unitary and orthogonal cases, the form is preserved; in the symplectic case, it only needs to be preserved up to a multiplier.  We set $f=2$ if $H$ is unitary; otherwise $f=1$.
In what follows, we will write  $H = \Cl(V)$ to specify that $H$ is one of the described groups.

By a \emph{classical group} $G$ (\emph{in dimension $n$}, if we wish to specify),
we mean henceforth  a group which is the quotient of some group $H=\Cl(V)$ as above by a central subgroup $Z$ of $H$.
Note that $|Z| \le \max(q+1,2) < 2q$ and that $|H/[H,H]| < 2q$. For such a group $G$, the {\it rank} of $G$ is the semisimple
rank of the algebraic group underlying $H$.
The finite simple groups $G$ with which we are concerned are of this type, but for the purposes of \S7 it will be useful to do things in this slightly greater generality. Note that any classical group in our sense is a classical group in the sense of \cite[Definition 1.2]{GLT3}; hence
the results of \cite{GLT3} apply.

Theorem~\ref{epsilon} is obtained by combining recent estimates for values of irreducible characters of classical groups with the following
result, which may be of independent interest.

\begin{theor}\label{small-centralizer}
For every non-trivial word $w \in F_d$ there exists a constant $c = c(w)$ such that,
if $G$ is a classical group of rank $r$ over the field with $q$ elements,
and $g_1, \ldots , g_d \in G$ are chosen uniformly and independently, then
\[
 \Pr[|\CB_G(w(g_1,\ldots,g_d))| \le q^{cr}] \to 1 \;\; {\rm as} \;\; |G| \go \infty.
\]
\end{theor}

We now embark on the proof of Theorem \ref{small-centralizer}.
This result is trivial if the rank $r$ is bounded, so we may assume $G$ is classical
of arbitrarily high rank.
We follow \cite{LS3} closely.

\begin{lem}
If $h\in H$ maps to $g\in G$, with $G = H/Z$ as above, then  $|\CB_G(g)| \leq |\CB_{H}(h)|.$
\end{lem}
\begin{proof}
Let
$$J = \{j\in H\mid j^{-1} h j \in hZ\}.$$
Then $J$ is a group containing $Z$, and $x\mapsto j^{-1} h j h^{-1}$ defines a homomorphism $J\to Z$ whose kernel is
$\CB_H(h)$.  It follows that $|J| \le |\CB_H(h)|\cdot |Z|$.  The restriction of the quotient map $H\to G$ to $J$ has kernel $Z$ and image
$\CB_G(g)$.  Thus,
$$|\CB_G(g)| = |Z|^{-1}|J| \leq |\CB_H(h)|.$$
\end{proof}

\begin{lem}
\label{GtoH}
For any $A \in \R_{>0}$,
\begin{equation*}
\begin{split}
&\frac{|\{(g_1,\ldots,g_d) \in G^d:\; |\CB_G(w(g_1,\ldots,g_d))| > A\}|}{|G|^d} \\
&\qquad\qquad\qquad\leq \frac{|\{(h_1,\ldots,h_d)\in {H}^d:\; |\CB_H(w(h_1,\ldots,h_d))|>A\}|}{|H|^d}.
\end{split}
\end{equation*}
\end{lem}

\begin{proof}
Indeed, any preimage in $H^d$ of an element $(g_1,\ldots,g_d)$ in the left-hand side numerator
belongs to the set in the right-hand side numerator.
The lemma follows.
\end{proof}

Equivalently,
$$\Pr[|\CB_G(w(g_1,\ldots,g_d))| > A] \leq  \Pr[|\CB_H(w(h_1,\ldots,h_d))| > A].$$%
Therefore, to prove that there exists $c >0$ such that
$$\limsup_{|G|\to \infty} \Pr[|\CB_G(w(g_1,\ldots,g_d))| > q^{cr}] = 0$$
it suffices to prove that there exists $c>0$ such that
$$\limsup_{|H|\to \infty} \Pr[|\CB_H(w(h_1,\ldots,h_d))| > q^{cr}] = 0,$$
so it is certainly enough to prove there exists $c > 0$ such that
$$\limsup_{|H|\to \infty} \Pr[|\CB_{\GL(V)}(w(h_1,\ldots,h_d))| > q^{cr}] = 0.$$

If $\F$ is a finite field and $h\in \GL_n(\F)$, we define for each monic irreducible polynomial $P(x)\in \F[x]$
$$a_{P,1}\ge a_{P,2}\ge \ldots $$
to be the descending sequence giving the sizes of Jordan blocks  for any root $\lambda$ of $P(x)$.  (Clearly,
this sequence does not depend on the choice of root $\lambda$.)
Clearly,
\begin{equation}
\label{total-blocks}
\sum_{P,m}a_{P,m}\deg P = n.
\end{equation}

It is well known
\cite[\S 1.3]{Hu} that the centralizer of $h$ in $M_n(\F)$ is a vector space over $\F$
of dimension
$$\sum_P \sum_m (2m-1)a_{P,m}\deg P.$$
Thus,
$$|\CB_{\GL_n(\F)}(h)| < |\F|^{\sum_P \sum_m (2m-1)a_{P,m}\deg P}.$$
For later use, we note that by (\ref{total-blocks}), if
$$|\CB_{\GL_n(\F)}(h)|>|\F|^{2\delta n^2},$$
then
\begin{equation}\label{cond3}
  \mbox{there exist some }P\mbox{  and some }m_0 > \delta n\mbox{ such that }a_{P,m_0}\neq 0,
\end{equation}
i.e., some eigenspace of $h$ has dimension greater than $\delta n$.
For immediate use, we note that
$$|\CB_{\GL_n(\F)}(h)| > |\F|^{6cn}$$
implies
\begin{align*}
\sum_P \sum_{m>c}(m-c)&a_{P,m}\deg P = -cn+ \sum_P \sum_{m>c}m a_{P,m}\deg P  \\
								&> -cn + \sum_P \sum_{m\ge 1}m a_{P,m}\deg P - \sum_P\sum_{m=1}^c \sum_{k=m}^c a_{P,k}\deg P\\
								&\ge -cn + \sum_P \sum_{m\ge 1}m a_{P,m}\deg P - c\sum_P \sum_{k=m}^c a_{P,k}\deg P\\
								&> -2cn+\frac 12 \sum_P \sum_{m\ge 1} (2m-1)a_{P,m}\deg P > cn
\end{align*}
As $a_{P,m}$ is non-increasing in $m$,
$$\sum_{m>c}(m-c)a_{P,m}\deg P \leq \biggl(\max_{\{(m,P)\mid a_{P,m}>0\}}(m-c)\biggr) \sum_P a_{P,c+1}\deg P.$$
Thus, at least one of the following conditions holds:
\begin{equation}\label{cond1}
  \sum_P a_{P,c+1}\deg P>\sqrt{cn},
\end{equation}
or
\begin{equation}\label{cond2}
\mbox{for some polynomial }P\mbox{ and some }m_0>\sqrt{cn},\mbox{ we have }a_{P,m_0}\neq 0.
\end{equation}

\begin{lem}\label{centralizer}
Condition \eqref{cond3} implies that there exists
a non-constant polynomial $Q(x)\in \F[x]$ such that
$$\dim_\F \ker Q(g) > \delta n\deg Q.$$
For any positive integer $t>0$, if $c$ is sufficiently large in terms of $t$ and $n$ is sufficiently large in terms of $c$, then
the conditions \eqref{cond1} and  \eqref{cond2}  each imply that there exists
a non-zero polynomial $Q(x)\in \F[x]$ such that
$$\dim_\F \ker Q(g) > 2t\deg Q+\sqrt{n}.$$
\end{lem}

\begin{proof}
If \eqref{cond3} holds, setting $Q = P$, we have that
$$\dim\ker Q(g) \ge m_0\deg P >  \delta n\deg P.$$

In case of \eqref{cond1}, $a_{P,1}\ge a_{P,2}\ge \ldots \ge a_{P,c+1}$ for all $P$ implies
$$\sum_P a_{P,c+1}\deg P \leq \frac{n}{c+1}.$$
Assuming $c\ge 2t$ and $n>(c+1)^2$, we set $Q = \prod P^{a_{P,c+1}}$ and obtain $\deg Q>\sqrt n$.
Regarding $\F^n$ as an $\F[x]$-module, where $x$ acts as $g$, the kernel of $Q(g)$ is isomorphic to
$$\bigoplus_P \Bigl[(\F[x]/(P(x)^{a_{P,c+1}}))^{c+1} \oplus \bigoplus_{m>c+1} \F[x]/(P(x)^{a_{P,m}})\Bigr],$$
whose dimension is $\ge (c+1)\deg Q> 2t\deg Q+\sqrt n$.

In case of \eqref{cond2}, if $c>(2t+1)^2\ge 1$ then we have
$\deg P \le n/m_0 < \sqrt n$.  Setting $Q = P$,
we obtain
$$\dim\ker Q(g) \ge m_0\deg P >  \sqrt{cn}\deg P > (2t+1)\sqrt n\deg P > 2t\deg P + \sqrt n.$$
\end{proof}

\begin{propo}
\label{big-kernel}
Let $H=\Cl(V)$ be a finite classical group as described,
where $\F = \F_{q^f}$ and $n=\dim_{\F}V$.  Let $w\in F_d$ be a word of length $l>0$.  If  $k$ and $D$ are positive integers
and $(h_1,\ldots,h_d)$ is chosen uniformly from $H^d$, the probability that there exists a polynomial $Q(x)\in \F[x]$ of degree $D$ such that
$$\dim_{\F}\ker Q(w(h_1,\ldots,h_d)) \ge 2lDk$$
is at most $q^{-fk((k-1)lD-2.5)}$.
\end{propo}

\begin{proof}
We choose an ordered $k$-tuple $(v_1,\ldots,v_k)$ uniformly from $V^k$ and a $d$-tuple $(h_1,\ldots,h_d)$ uniformly from $H^d$.
It suffices to prove that the probability that $Q(w(h_1,\ldots,h_d))(v_i) = 0$ for all $i\in [1,k]$ is less than
$q^{f(2.5k+k(k+1)lD-kn)}$.
Indeed, the probability that a uniformly chosen random $k$-tuple $(v_1,\ldots,v_k)$ of vectors belongs to any particular subspace of dimension $2lDk$ is $q^{2flDk^2-fkn}$, so this implies that the probability that  $\dim Q(w(h_1,\ldots,h_d)) \ge 2lDk$ is at most $q^{-fk((k-1)lD-2.5)}$, as claimed.

We write $w^D$ as a reduced word $y_m y_{m-1} \ldots  y_1$, where $m \leq lD$ and each $y_i$ belongs to $\{x_1^{\pm 1},\ldots,x_d^{\pm 1}\}$.
Let $z_j = y_j y_{j-1}\cdots  y_1$ and for $j\ge 0$, let
$$e_{i,j} = z_j(h_1,\ldots,h_d)(v_i).$$
If $\{v_1,\ldots,v_k\}\subset \ker Q(w(h_1,\ldots,h_d))$, then for each $i\in [1,k]$, the set $\{e_{i,0},e_{i,1},\ldots,e_{i,m}\}$ is linearly dependent.

We endow the set of integer pairs in $[1,k]\times [0,m]$ with the lexicographic ordering.
Let the event $X_{i,j}$ be the condition
$$e_{i,j}\not\in \Span \{e_{i',j'}\mid (i',j') < (i,j)\}.$$
Let
$$Y_{i,j} = X_{i,0}X_{i,1}\cdots X_{i,j-1}X_{i,j}^c.$$
Let $Z_i$ be the event that $e_{i,0},e_{i,1},\ldots,e_{i,m}$ is a linearly independent sequence.  If $Z_i^c$ occurs, then $Y_{i,j}$ occurs for some $j\in [0,m]$.
Thus,
$$\Pr[Z_i^c | Z_1^cZ_2^c\cdots Z_{i-1}^c] \le \sum_{j=0}^m \Pr[Y_{i,j} | Z_1^cZ_2^c\cdots Z_{i-1}^c].$$
We find an upper bound for each term on the right hand side
by giving an upper bound on the conditional probability of $Y_{i,j}$ with respect to any possible set of data $e_{i',j}$ for $i'\in [1,i)$ and $j'\in [0,m]$.

Given this data, the event $Y_{i,0}$, or, equivalently, $X_{i,0}^c$, is the condition that $v_i$ belongs to the span of $\{e_{i',j'}\mid i' < i, 0\le j'\le m\}$.
As
$$\dim \Span\{e_{i',0},\ldots,e_{i',m}\} \le m,$$
the probability of $X^c_{i,0}$ is at most $q^{f((i-1)m-n)}$.
For $j\in [1,m]$, we further condition on $e_{i,0}, e_{i,1}, \ldots, e_{i,j-1}$ consistent with $X_{i,j'}$ for $j'< i$.
Either $y_i = x_t$ or $y_i = x_t^{-1}$ for some $t$, and $e_{i,j}$ is determined
by the specified value $e_{i,j-1}$ and the random variable $h_t$, so the conditional probability in question depends only $h_t$.

%
Let $h$ stand for $h_t$ if $y_i = x_t$ and for $h_t^{-1}$ if $y_i=x_t^{-1}$.
Let $w_1,\ldots,w_r$ and $w'_1,\ldots,w'_r$ denote two linearly independent sequences of vectors.
If $h$ is a uniformly distributed random variable on $H$, and we condition on $h(w_j) = w'_j$ for
$j=1,\ldots,r-1$, then the probability that $h(w_r) = w'_r$ is the reciprocal of the number of possibilities 
for $h(w_r)$ given that $h(w_j)=w'_j$ for $j=1,\ldots,r-1$.

If $H$ is of linear type, it contains $\SL(V)$ and therefore acts transitively on $r$-tuples of linearly independent vectors of $V$ for $r < n$.
It follows that
any $w'$ not in the span of $w'_1,\ldots,w'_{r-1}$ is possible.
Otherwise $H$ contains $\SU(V)$, $\Sp(V)$, or $\Omega(V)$, respectively.
In each of these cases, Witt's extension theorem \cite[Proposition 2.1.6]{KlL} applied to
$\tilde H = \mathrm{U}(V)$, $\Sp(V)$, or $\mathrm{O}(V)$, respectively, implies that the number of possibilities for $h(w_r)$ is at least $1/\alpha$ of
the number of solutions in $w'$ of the system of equations
\begin{equation}
\label{witt}
\langle w'_j,w'\rangle = \langle w_j,w_r\rangle,\,j=1,\ldots,r-1;\;\langle w',w'\rangle = \langle w_r,w_r\rangle,
\end{equation}
where $\alpha=q+1$ in the $\mathrm{U}$-case, $1$ in the $\Sp$-case, and $\alpha = 2$ or $4$ in the $\Or$-case, depending on whether $2|q$ or not.

The equations (\ref{witt}) are $\F_{q^f}$-linear except for the last, in which the left hand side is a quadratic form over $\F_q$.  Since a quadratic form
in $k$ variables over a field of cardinality $q$ takes on each possible value at least $q^{k-2}$ times,  we conclude that the probability
of any single possible value $w'$ for $h(w_r)$ is at most $\alpha q^{f(1+r-n)}\le \alpha q^{f(1+(i-1)m+j-n)}$.
Since
$$\dim \Span \{e_{i',j'}\mid (i',j') < (i,j)\}\le (i-1)m+j,$$
we conclude that
$$\Pr[Y_{i,j} | Z_1^cZ_2^c\cdots Z_{i-1}^c] \le \alpha q^{f((i-1)m+j)} q^{f(1+(i-1)m+j-n)} = \alpha q^{f(1+2(i-1)m+2j-n)}.$$
It follows that
\begin{align*}
\Pr[Z_i^c\mid Z_1^c Z_2^c \cdots  Z_{i-1}^c]  &\le \sum_{j=0}^m  \alpha q^{f(1+2(i-1)m+2j-n)} \\
	&< \frac \alpha{1-q^{-2f}}q^{f(1+2im-n)} < q^{f(2.5+2im-n)}.
\end{align*}
This implies
$$\Pr[Z_1^c\cdots  Z_k^c] \le q^{f(3k+k(k+1)m-kn)} \le q^{f(2.5k+k(k+1)lD-kn)},$$
as claimed.
\end{proof}

\begin{proof}[Proof of Theorem \ref{small-centralizer}]
This follows by combining Lemma \ref{centralizer} with Proposition \ref{big-kernel}. Indeed, by the discussion preceding
Lemma \ref{centralizer}, we need to bound from  above the probability $\Pr'$ that either \eqref{cond1} or \eqref{cond2}
holds for $h = w(h_1, \ldots,h_d)$. We may assume that the rank $r$ of $H$ is as large as
we wish; in particular, we may assume that
$$r_0 = \lfloor \sqrt[4]{r}/\sqrt{2l} \rfloor \geq 4,$$
where $l$ is the length of $w$.
First we apply Lemma \ref{centralizer} with $t= 2l$ to see that either of \eqref{cond1}, \eqref{cond2} for $h$ implies the existence
of non-constant $Q \in \F[x]$ such that
$$\dim_\F \ker Q(h) > 2t\deg Q+\sqrt{n} > \max(4l\deg Q,2lr_0^2).$$
By Proposition \ref{big-kernel} applied to $k= 2$, the probability $\Pr'_1$ that $D = \deg Q$ is at least $r_0$ is
$$\Pr'_1 < q^{5f}\sum^{\infty}_{D=r_0}q^{-2flD} < \frac{q^{5f}}{q^{2r_0f}(1-q^{-2f})} < q^{-r_0f/2}.$$
On the other hand, by Proposition \ref{big-kernel} applied to $k= r_0 \geq 4$, the probability $\Pr'_2$ that $D = \deg Q \geq 1$ is $<r_0$ is
$$\Pr'_2 < q^{2.5r_0f}\sum^{\infty}_{D=1}q^{-3r_0flD} < \frac{q^{2.5r_0f}}{q^{3r_0f}-1} < 2q^{-r_0f/2}.$$
Note that when $|G| \to \infty$, we have that $q^{r_0} \to \infty$, and so
$$\Pr' = \Pr'_1+\Pr'_2 < 3q^{-r_0f/2}$$
tends to $0$, as desired.
\end{proof}

\begin{proof}[Proof of Theorem \ref{epsilon}]
By \cite[Theorem~1.3]{GLT2} and \cite[Theorem~1.3]{GLT3}, for all $c$ and $\epsilon > 0$, increasing $r$ if necessary,
$|\CB_G(g)| < q^{cr}$ implies
$$|\chi(g)| \leq \chi(1)^\epsilon$$
for every irreducible character $\chi$ of $G$.  Now apply Theorem \ref{small-centralizer}.
\end{proof}

\begin{proof}[Proof of Theorem \ref{prob-waring}]

It remains to show that, given any two disjoint words $w_1,w_2 \neq 1$,
there exists a positive constant $R$ such that if $S$ is any set of finite simple groups of rank $r \ge R$, and $w = w_1w_2$, then
$$\lim_{G \in S, |G| \to \infty} \Vert p_{w,G}- U_G \Vert_{L^1}= 0.$$
Fix any $0 < \epsilon < 1/3$.
We say that an element $g\in G$ is \emph{$\epsilon$-good} if $\chi(g)|\le \chi(1)^\epsilon$ for all $\chi\in\Irr(G)$; a conjugacy class $C$ is \emph{$\epsilon$-good}
if it consists of $\epsilon$-good elements.
By Theorem \ref{epsilon}, if $R$ is chosen sufficiently large, $G$ is of rank $r>R$, and $g_1,\ldots,g_d\in G$ are chosen uniformly and independently,
then the probability that $w_1(g_1,\ldots,g_d)$ and $w_2(g_1,\ldots,g_d)$ are both $\epsilon$-good approaches $1$ as $|G|\to \infty$.
For proving $L^1$ convergence to the uniform distribution, we may therefore assume that
both $w_1(g_1,\ldots,g_d)$ and $w_2(g_1,\ldots,g_d)$ belong to $\epsilon$-good conjugacy classes.

For  $i\in \{1,2\}$ and any conjugacy class $C_i$, the conditional distribution of $w_i(g_1,\ldots,g_d)$ given that it belongs to $C_i$
is the uniform distribution on $C_i$.
Thus, it suffices to prove that
the convolution of the uniform distribution on an $\epsilon$-good $C_1$ with the uniform distribution on an $\epsilon$-good $C_2$ approaches the uniform distribution on $G$ in the $L^1$ norm
uniformly in $C_1$ and $C_2$ as the order of $G$ (of sufficiently high rank) grows without bound. This would follow if we knew that there exist at least $(1-o(1))|G|$ elements $g \in G$ for which the probability
that $x_1x_2=g$ as $x_i \in C_i$ are chosen uniformly and independently is  $(1+o(1))|G|^{-1}$, where $o(1) = o_{|G|}(1)$.

By Proposition 4.2 of \cite{LS4}, if $\e>0$ and the rank of $G$ is sufficiently large in terms of $\epsilon$, then the proportion of $\epsilon$-good elements in $G$   tends to $1$ as $|G| \to \infty$.
Hence we may assume that $g$ is $\epsilon$-good.
By (\ref{Frobenius}), the probability $p(C_1,C_2,g)$ that $x_1 x_2 = g$ satisfies
\[
p(C_1,C_2,g) = |G|^{-1} \sum_{\chi \in \Irr(G)} \frac{\chi(C_1)\chi(C_2)\chi(g^{-1})}{\chi(1)},
\]
so
\[
|p(C_1,C_2,g)-|G|^{-1}| \leq |G|^{-1} \sum_{1_G \ne \chi \in \Irr(G)} \frac{\chi(1)^{3 \e}}{\chi(1)} = |G|^{-1} (\zeta_G(1-3\e)-1),
\]
where $\zeta_G(s) = \sum_{\chi \in \Irr(G)} \chi(1)^{-s}$ is the Witten zeta function of $G$.
Since $1 - 3\e > 0$, we may choose $R$ sufficiently large so that $\zeta_G(1-3\e) \to 1$ as $|G| \to \infty$;
indeed, this follows from Theorem 1.1 and 1.2 of  \cite{LiSh4}.
This yields
\[
p(C_1,C_2,g) = |G|^{-1}(1+o(1)),
\]
for all $C_1,C_2,g$ as above. This completes the proof of Theorem \ref{prob-waring}.
\end{proof}

\section{The $L^\infty$ Waring problem}

In this section we prove Theorem \ref{main-inf}.

\begin{propo}\label{bound1}
Fix any $0 < \e < 1$. There exists some $A(\e) > 0$ such that, for any $n \geq A(\e)$, any element $g$ in $G \in \{ \AAA_n,\SSS_n\}$,
and any $\chi \in \Irr(G)$, the following two statements hold.
\begin{enumerate}[\rm(i)]
\item If $\Fix(g) \leq n^{1-\e}$, then $|\chi(g)| \leq \chi(1)^{1-\e/3}$.
\item If $\Fix(g) = k$, then $|\chi(g)| \leq n^{k/4}\chi(1)^{1/2+\epsilon}$.
\end{enumerate}
\end{propo}

\begin{proof}
(a) First we consider the case $G = \SSS_n$. Then  \cite[Theorem 1.2]{LS1} implies (i) immediately. For (ii), it implies that
there exists $C(\e) > 0$ such that, if $n-k \geq C(\e)$ and $x \in \SSS_{n-k}$ is fixed-point-free, then  \begin{equation}\label{fpf}
  |\chi(x)| \leq \chi(1)^{1/2+\epsilon}
\end{equation}
for all $\chi \in \Irr(\SSS_{n-k})$. On the other hand,
$$\frac{|\chi(g)|}{n^{k/4}\chi(1)^{1/2}} \leq \frac{\chi(1)^{1/2}}{n^{k/4}} \leq \frac{|\SSS_n|^{1/4}}{n^{k/4}} < \frac{(n/2)^{n/4}}{n^{k/4}}
   = \left( \frac{n^{n-k}}{2^n} \right)^{1/4}.$$
In particular, the desired bound in (ii) holds if $n-k < C(\e)$ is bounded, but $n$ is large enough.

Hence we may assume that $n-k \geq C(\e)$ is sufficiently large, and also $\e < 1/2$.
We use the branching rule from $\SSS_m$ to $\SSS_{m-1}$ for $n \geq m \geq n-k+1$
consecutively and write
$$\chi|_{\SSS_{n-k}} = \chi_1+\cdots +\chi_N$$
of irreducible characters of $\SSS_{n-k}$, with repetition allowed. The number of terms $N$ is at most the $k^{\mathrm {th}}$ power of the maximum number of removable boxes from any Young diagram of size $\leq n$, and so $N < (2n)^{k/2}$.

Let $h\in \SSS_{n-k}$ map to an element of $\SSS_n$ conjugate to $g$. Then $\chi(g) = \sum_i \chi_i(h)$. Since $\Fix(g) = k$,
$h$ has no fixed point; also, $n-k \geq C(\e)$. Hence  $|\chi_i(h)| \leq \chi_i(1)^{1/2+\epsilon}$ by \eqref{fpf}.
As $\e < 1/2$, we obtain
$$\begin{aligned}|\chi(g)| & \leq \sum_i |\chi_i(h)| \leq \sum_i \chi_i(1)^{1/2+\epsilon} \leq N \left( \frac{\sum_i \chi_i(1)}{N} \right)^{1/2+\e}\\
    & = N \left(\frac{\chi(1)}{N}\right)^{1/2+\epsilon} < (2n)^{k/4-k\e/2}\chi(1)^{1/2+\e} \leq n^{k/4}\chi(1)^{1/2+\epsilon}\end{aligned}$$
when $n > 2^{1/2\e}$.

\smallskip
(b) Now we consider the case $G = \AAA_n$.
We are certainly done by (a) if $\chi$ extends to $G$. Hence we may assume that there is a
self-associated partition
$$\lambda = (\lambda_1 \geq \lambda_2 \geq \ldots \geq \lambda_r \geq 1)$$
of $n$ such that the character $\chi^\lambda$ of $\SSS_n$ labeled by $\lambda$
restricts to $G$ as $\chi^+ + \chi^-$, with $\chi^\pm \in \Irr(G)$ and $\chi = \chi^+$. Let $h_{11}> \ldots > h_{tt} \geq 1$ denote the
hook lengths of the Young diagram of $\lambda$ at the diagonal nodes. By \cite[Theorem 2.5.13]{JK}, if the cycle type of
$g$ is not $(h_{11}, h_{22}, \ldots,h_{tt})$, then
$$|\chi(g)| = \frac{|\chi^\lambda(g)|}{2} \leq \frac{\chi^\lambda(1)^{1-\e/3}}{2} \leq \chi(1)^{1-\e/3}$$
if $\Fix(g) \leq n^{1-\e}$; and the same argument applies to (ii).
On the other hand, if the cycle type of $g$ is $(h_{11}, h_{22}, \ldots,h_{tt})$, then
$$|\chi(g)| \leq \left(1+\sqrt{\prod^t_{i=1}h_{ii}}\right)/2.$$
Since $\lambda$ is self-associated, we have that $\lambda_1 \leq (n+1)/2$, and so
\begin{equation}\label{deg1}
 \chi(1) = \chi^\lambda(1)/2 \geq 2^{(n-5)/4}
\end{equation}
by \cite[Theorem 5.1]{GLT1}. Also, all $h_{ii}$ are odd integers. Note that if $m \geq 3$ is any odd integer,
then
$$m < 2^{m/4}$$
if and only if $m \geq 17$; furthermore,
$$\prod^{7}_{j=1}\frac{2j+1}{2^{(2j+1)/4}} < 37.$$
It follows that
$$\prod^t_{i=1}h_{ii} < 37 \cdot 2^{\sum^t_{i=1}h_{ii}/4} = 37 \cdot 2^{n/4}$$
and so
$$|\chi(g)| < (1+6.1 \cdot 2^{n/8})/2.$$
Together with \eqref{deg1}, this implies that
$$|\chi(g)| < \min\{\chi(1)^{2/3},\chi(1)^{1/2+\e}\} \leq \min\{\chi(1)^{1-\e/3},n^{k/4}\chi(1)^{1/2+\e}\}$$
when $n$ is large enough.
\end{proof}

For use in \S7, we will also need an imprimitive version of Proposition \ref{bound1}:

\begin{lem}
\label{wr}
For any $0 < \e < 1$, let $A(\e) > 0$ be the constant in Proposition \ref{bound1}. Let $m|n$  and consider the subgroup
$H = \SSS_{n/m} \wr \SSS_m$ of $G = \SSS_n$. If $n/m \geq A(\e)$, then for any $h \in H$ and $\chi \in \Irr(H)$ we have
\begin{enumerate}[\rm(i)]
\item If $\Fix(h) \le (n/m)^{1-\epsilon}$, then $|\chi(h)| \le m! \chi(1)^{1-\epsilon/3}$.
\item  If $\Fix(h) = k$, then $|\chi(h)|  \le m!(n/m)^{k/4}\chi(1)^{1/2+\epsilon}$.
\end{enumerate}
\end{lem}

\begin{proof}
Let $K = (\SSS_{n/m})^m\triangleleft H$.  The restriction of any irreducible representation $V$ of $H$ to $K$
is a direct sum of representations of the form $V_1\boxtimes \cdots \boxtimes V_m$,
where each $V_i$ is an irreducible representation of $\SSS_{n/m}$, and the $m$-tuples
$(V_1,\ldots,V_m)$ appearing in tensor decompositions of the different irreducible factors of $V|_K$ are the same up to permutation.

There exist a unique partition  $\pi = (a_1, a_2, \ldots ,a_r) \vdash m$, and an irreducible representation $W_i$ of $\SSS_{n/m}$ for each $i$
(so that $W_1, \ldots ,W_r$ are pairwise non-isomorphic) such that, after permuting tensor factors, $V_1\boxtimes \cdots\boxtimes V_m$ can be rewritten
$$W_\pi = W_1^{\boxtimes a_1}\boxtimes\cdots\boxtimes W_r^{\boxtimes a_r}.$$
This representation has inertia group $H_\pi = \SSS_{n/m}\wr \SSS_\pi$ in $H$, where $\SSS_{a_1}\times \cdots\times \SSS_{a_r}\subset \SSS_m$,
$W_\pi$ extends to a representation $\tilde W_\pi$ of $H_\pi$, and
$$V = {\mathrm {Ind}}^H_{H_\pi}(\tilde W_\pi \otimes U)$$
for a suitable irreducible representation $U$ of $\SSS_\pi$ (inflated to $H_\pi$).

To calculate the trace $\chi(h)$ of $h\in H$ acting on $V$, we first consider the image $\bar h$ of $h$ in $\SSS_m$.  In general,
$\bar h$ permutes the $\SSS_\pi$-cosets and therefore the summands of $V|_{H_\pi}$.  Only the summands which are stabilized by
$\bar h$ contribute to $\chi(h)$, and the number of those summands is certainly bounded by $[\SSS_m:\SSS_\pi]$.
Also, the absolute value of the trace of $h$ acting on $U$ is at most $|\SSS_\pi|$.  Hence, it suffices to prove that assuming $\bar h\in \SSS_\pi$,
we have
\begin{enumerate}[\rm(i)]
\item If $\Fix(h) \le (n/m)^{1-\epsilon}$, then $\tr(h|\tilde W_\pi) \le \dim \tilde W_\pi^{1-\epsilon/3}$.
\item  If $\Fix(h) = k$, then $\tr(h|\tilde W_\pi) \le (n/m)^{k/4}\dim \tilde W_\pi^{1/2+\epsilon}$.
\end{enumerate}

Writing $h = (h_1,\ldots,h_r)$ where $h_i\in \SSS_{n/m}\wr S_{a_i}$ acts on the extension $\widetilde{W_i^{\boxtimes a_i}}$ of  $W_i^{\boxtimes a_i}$ to $\SSS_{n/m}\wr S_{a_i}$,
it suffices to prove
\begin{enumerate}[\rm(i)]
\item If $\Fix(h_i) \le (n/m)^{1-\epsilon}$, then $\tr(h_i|\widetilde{W_i^{\boxtimes a_i}}) \le \dim \widetilde{W_i^{\boxtimes a_i}}^{1-\epsilon/3}$.
\item  If $\Fix(h_i) = k$, then $\tr(h_i|\widetilde{W_i^{\boxtimes a_i}}) \le (n/m)^{k/4}\dim \widetilde{W_i^{\boxtimes a_i}}^{1/2+\epsilon}$.
\end{enumerate}

Thus, we can reduce to the case $r=1$.
Decomposing $\bar h \in \SSS_{a_1}$ into a product of disjoint cycles, we further reduce to the case that
$\bar h$ is an $a_1$-cycle $\sigma$, say $(1~2~\ldots ~a_1)$. Writing $h = ((t_1,...,t_{a_1}),\sigma)$ with $t_i \in \SSS_{n/m}$, we then
obtain
$$|\tr(h \mid W_1^{\boxtimes a_1})| = |\tr((t_1 \cdots t_{a_1})\mid W_1)|;$$
in particular, it is at most $\dim W_1 \leq (\dim W_1^{\boxtimes a_1})^{1/2}$ if $a_1 \geq 2$, implying both (i) and (ii).
If $a_1 = 1$,
then $\Fix(t_1 \cdots t_{a_1}) = \Fix(h)$, and we are again done by Proposition \ref{bound1}(ii) applied to $\SSS_{n/m}$.
\end{proof}

Let $w\in F_d$ be a non-trivial word.  We write it in reduced form: $w= y_l\cdots  y_2 y_1$, where each
$y_i$ can be regarded as a function $G^d\to G$ which is either projection on the $k^{\mathrm {th}}$ factor for some $k\in [1,d]$
or projection composed with the inverse map.

\begin{lem}
\label{counting}
For integers $1\le a < b \le l$, $m\ge 1$, and $n > 2m(b-a)$, and $G\in \{\AAA_n,\SSS_n\}$, we define
$$X_{m,n}(w,a,b) \subset G^d\times [1,n]^{m(b-a)}$$
to be the set of tuples
$$(g_1,\ldots,g_d,r_{1,a},\ldots,r_{1,b-1},\ldots, r_{m,a},\ldots,r_{m,b-1})$$
satisfying:
\begin{enumerate}
\item[{\rm (\ref{counting}.1)}] all the $r_{i,j}$, $1\le i\le m$, $a\le j\le b-1$, are pairwise distinct;
\item[{\rm (\ref{counting}.2)}] for all $1\le i\le m$ and $a\le j \le b-2$,  $y_j(g_1,\ldots,g_d)(r_{i,j}) = r_{i, j+1}$; and
\item[{\rm (\ref{counting}.3)}] for all $1\le i\le m$, $y_{b-1}(g_1,\ldots,g_d)(r_{i,b-1}) = r_{i,a}$.
\end{enumerate}
Then the projection $p_1$ of $X_{m,n}(w,a,b)$ onto $G^d$ has cardinality less than or equal to
$$\frac{e^{2m^2(b-a)^2/n}}{m!}|G|^d .$$
\end{lem}

\begin{proof}If $y_a = y_{b-1}^{-1}$, then (\ref{counting}.2) for $j=a$ together with  (\ref{counting}.3) implies $r_{i,a+1} = r_{i,b-1}$ for all $i$,
contrary to (\ref{counting}.1), so it follows that $X_{m,n}(w,a,b)$ is empty.  We therefore assume $y_a \ne y_{b-1}^{-1}$.

For each choice of pairwise distinct $r_{i,j}$, $1\le i\le m$, $a\le j \le b-1$, the conditions (\ref{counting}.2)
and (\ref{counting}.3) impose a total of
$m(b-a)$ conditions on the $(g_1,\ldots,g_d)$, where each condition is of the form $g_h u = v$ or $g_h^{-1} u = v$,
for some $h\in [1,d]$ and $u,v\in [1,n]$.  These conditions are independent because the $r_{i,j}$ are all distinct from one another,
and $y_a \neq y_{b-1}^{-1}$.
Let $c_h$ for $1\le h\le d$ denote the number of conditions on $g_h$.
As $c_1+\cdots  + c_d = m(b-a) \le n-2$, each $c_h$ is less than $n-1$, and the number of elements  $g_h\in G$ satisfying
the $c_h$ conditions is
$$\frac{|G|}{n(n-1)\cdots  (n-c_h+1)} \ge \frac{|G|}{(n-c_h)^{c_h}}.$$
Overall,
$$|X_{m,n}(w,a,b)| \le n^{m(b-a)}\prod_{h=1}^d \frac{|G|}{(n-c_h)^{c_h}} \le \frac{n^{m(b-a)}}{(n-m(b-a))^{m(b-a)}}|G|^d .$$
The projection of $X_{m,n}(w,a,b)$ onto $G^d$ is at least $m!$ to $1$ since $\SSS_m$ acts faithfully on $X_{m,n}(w,a,b)$  through its
action on the $i$-coordinate of $r_{i,j}$.  Thus, the cardinality of the projection is bounded above by
$$\frac{n^{m(b-a)}}{m!(n-m(b-a))^{m(b-a)}}|G|^d.$$
Setting $M = m(b-a) < n/2$, we have
$$\frac{n^M}{(n-M)^M} = \exp\Bigl(M\log\Bigl(1+\frac M{n-M}\Bigr)\Bigr) < \exp\Bigl(\frac{M^2}{n-M}\Bigr) < \exp\Bigl(\frac{2M^2}n\Bigr),$$
which implies the lemma.
\end{proof}

For $G \in \{\AAA_n,\SSS_n\}$ and $w$ a word of length $l$ in $F_d$, let $W_n$ be the corresponding random variable on $G$
with distribution $p_{w,G}$.

\begin{propo}\label{bound2}
If $k \ge (el)^{36}$, then for all positive integers $n$,
$$\Pr[\Fix(W_n) \ge k] \leq k^{-\frac{k}{3l^4}}.$$
\end{propo}

\begin{proof}
Let $X_i$, $i=1,\ldots,d$, be independent uniform random variables on $G$.
We write $w=y_l \cdots y_2y_1$ in reduced form, and let $z_i = y_i \cdots  y_2 y_1$.
Let $Y_i$ denote the random variable $y_i(X_1,\ldots,X_d)$.

Let us first assume that $l^4$ divides $k$ and if $l=1$ we assume also $k\le n/2$.
Since there are less than $l^2$ pairs of integers $(a,b)$ with $0\le a < b\le l$,
if $\Fix(w(g_1,\ldots,g_d)) \ge k$, there exist $a$ and $b$ and at least $k/l^2$ integers $r\in[1,n]$  such that the following two conditions hold:
\begin{enumerate}
\item[(\ref{bound2}.1)] the terms of the sequence
$$z_a(g_1,\ldots,g_d) r,z_{a+1}(g_1,\ldots,g_d)r,\ldots,z_{b-1}(g_1,\ldots,g_d)r$$
are pairwise distinct, and
\item[(\ref{bound2}.2)] $z_a(g_1,\ldots,g_d)r = z_b(g_1,\ldots,g_d)r$.
\end{enumerate}

For $a\le i,j < b$ and any given $r\in [1,n]$, there is at most one element $s\in \Fix (w(g_1,\ldots,g_d))$ such that
$$z_i(g_1,\ldots,g_d) r =  z_j(g_1,\ldots,g_d) s.$$
Thus, if there exist $k/l^2$ elements $r$ satisfying (\ref{bound2}.1) and (\ref{bound2}.2), there exists
a subset of $m = k/l^4$ elements $\{r_1,\ldots,r_m\}$
for which the sets
$$\{\{z_a(g_1,\ldots,g_d) r_j,z_{a+1}(g_1,\ldots,g_d)r_j,\ldots,z_{b-1}(g_1,\ldots,g_d)r_j\}\mid 1\le j \le m\}$$
are pairwise disjoint.  Setting
$$r_{j,i} = z_i(g_1,\ldots,g_d)(r_j),$$
we see that the tuple
$$(g_1,\ldots,g_d,r_{1,a},\ldots,r_{1,b-1},\ldots,r_{m,b-1})$$
satisfies conditions (\ref{counting}.1)--(\ref{counting}.3).
Thus, the set
$$\{(g_1,\ldots,g_d)\mid \Fix(w(g_1,\ldots,g_d)) \ge k\}$$
is contained in
$$\bigcup_{0\le a<b\le l}p_1(X_{m,n}(w,a,b)).$$
As $2ml = 2k/l^3 \le n$, Lemma~\ref{counting} applies, so
\begin{align*}\Pr[\Fix(W_n)\ge k]&=\Pr[\Fix(Y_l\cdots  Y_1)\ge k] \\
&\le \frac{l^2\max_{a,b} |p_1(X_{m,n}(w,a,b))|}{|G|^d}
\le \frac{l^2 e^{2k^2/l^6n}}{(k/l^4)!}.
\end{align*}

\smallskip
We now consider the general case $k > (el)^{36}$.
If $k_1$ denotes the largest multiple of $l^4$ not exceeding $k$ (or $n/2$ if $l=1$), we have $k_1 \geq k/2$, and
$$(k_1/l^4)! \geq (k_1/el^4)^{k_1/l^4} \geq (k/2el^4)^{k/2l^4}.$$
By (i), we now have
$$\begin{aligned}
-\log \Pr[\Fix(W_n) \geq k] & \ge -\log \Pr[\Fix(W_n) \geq k_1]\\
& \geq -2\log l - \frac{2k_1^2}{l^6 n}+\frac{k}{2l^4} \log \frac{k}{2el^4} \\
   & \geq -2\log l - \frac {2k}{l^6} - \frac {k}{2l^4} \log 2el^4 + \frac {k\log k}{2l^4}.
\end{aligned}
$$
As
$2\log l$, $\frac {2k}{l^6}$, $\frac {k}{2l^4} \log 2el^4$ all do not exceed $\frac {k\log k}{18l^4}$,
we conclude that
$$-\log \Pr[\Fix(W_n) \geq k] \ge \frac{k\log k}{3l^4},$$
as claimed.
\end{proof}

The following variant
of Proposition~\ref{bound2}, where $H$ is allowed to be any permutation group and $W_n$ is the random variable on $H$ corresponding to $w$,
is needed in \S7.  We say $H$ is \emph{$\epsilon$-roughly transitive} if for all $1\le i\le t = n^{1-\epsilon}$, the size of every $H$-orbit of
ordered $i$-tuples of pairwise distinct integers in $[1,n]$ is at least $t^i$.

\begin{propo}
\label{nearly-trans}
Let $l$ be a positive integer and $0<\epsilon < 1/4l$.  If $n$ is sufficiently large in terms of $l$,  $2el^4n^{1/2}<k < n^{1-\epsilon}$,
and $H<\SSS_n$ is $\epsilon$-roughly transitive, then
$$\Pr[\Fix(W_n)\ge k] \le n^{- k/8l^4}.$$
\end{propo}

\begin{proof}
The number of elements  $g\in H$ satisfying $c$ conditions of the form $g u = v$ or $g^{-1}u=v$
is at most $|H|$ divided by the cardinality of the smallest $H$-orbit of a $c$-tuple $(s_1,\ldots,s_c)$ of pairwise distinct integers in $[1,n]$, which is bounded above by
$t^{-c} |H|$.
Thus, following the notation of Lemma~\ref{counting},
$$|X_{m,n}(w,a,b)| \le n^{m(b-a)}\frac{|H|^d}{t^{m(b-a)}}\le (n/t)^{ml} |H|^d \le n^{\epsilon ml}|H|^d.$$
Let $k_1$ be the largest multiple of $l^4$ which is bounded above by $k$.
As in Proposition~\ref{bound2}, $\Fix(w(g_1,\ldots,g_d))\ge k_1$ implies there exist at least $k_1/l^2$ elements satisfying
(\ref{bound2}.1) and (\ref{bound2}.2) and therefore $(g_1,\ldots,g_d)$ lies in the projection of $X_{k_1/l^4,n}(w,a,b)$
for some $a,b$ with $1\le a<b\le l$.

As $\epsilon < 1/4l$ and $k/2el^4 > \sqrt n$, if $n$ is sufficiently large,
$$\Pr[\Fix(W_n)\ge k_1] \le \frac{l^2 n^{\epsilon k_1/l^3}}{(k_1/l^4)!} \le \frac{l^2 n^{\epsilon k/2l^3}}{(k_1/el^4)^{k_1/l^4}}\le \frac{l^2 n^{\epsilon k/2l^3}}{(k/2el^4)^{k/2l^4}}\le \frac{n^{k/8l^4}}{n^{k/4l^4}},$$
which gives the proposition.
\end{proof}

\begin{defi}\label{parity1}
{\em
\begin{enumerate}[\rm(i)]
\item Recall from the introduction that a word
$w$ in the free group $F_d$ is {\it even}
if $w \in \langle [F_d,F_d],x_1^2, \ldots ,x_d^2 \rangle$. Otherwise we say that $w$ is {\it odd}.
We define $\gamma(w) = 1$ (resp.\ $\gamma(w) = 0$) when $w$ is even (resp.\ odd).

\item For $G = \SSS_n$, let $U^0_G = U_G$, the uniform distribution on $G$, and let
$U^1(\{g\}) = 2/|G|$ for $g \in \AAA_n$ and $U^1(\{g\}) = 0$ for $g \in G \smallsetminus \AAA_n$.
\end{enumerate}}
\end{defi}

Note that if $w = w_1w_2 \cdots w_N$ is a product of pairwise disjoint words, then
\begin{equation}\label{for-gamma}
  \gamma(w) = \gamma(w_1)\gamma(w_2) \cdots \gamma(w_N).
\end{equation}
The relevance of Definition \ref{parity1} follows from the following statement, where $\sgn$ denotes the sign character of $\SSS_n$:

\begin{lem}\label{parity2}
Let $G = \SSS_n$, $w \in F_d$, and let $X$ be the random variable on $G$ with distribution $p_{w,G}$. Then
$$\sum_C \Pr(X \in C)\sgn(C) = \gamma(w),$$
where the summation runs over conjugacy classes $C$ in $G$.
\end{lem}

\begin{proof}
Note that the sum in question is $\Sigma=\Pr(X \in \AAA_n) - \Pr(X \notin \AAA_n)$. If $w$ is even, then $X$ is always in $\AAA_n$,
whence $\Sigma = 1=\gamma(w)$. If $w$ is odd, then half of the time $X$ belongs to $\AAA_n$ and half of the time it does not, whence
$\Sigma = 0 = \gamma(w)$.
\end{proof}

\begin{propo}
\label{waring-alt} Let $l, N \geq 1$ be integers such that $N>8l^4+41$.  If $w_1, w_2, \ldots , w_N$ is a sequence of non-trivial words of length at most $l$, then
we have
\[
\lim_{n \to \infty} \Vert \underbrace{p_{w_1,\AAA_n}\ast \cdots  \ast p_{w_N,\AAA_n}}_{N}- U_{\AAA_n} \Vert_{L^\infty}= 0.
\]
Furthermore, for $\gamma= \gamma(w_1w_2 \cdots w_N)$, we have
\[
\lim_{n \to \infty} \Vert \underbrace{p_{w_1,\SSS_n}\ast \cdots  \ast p_{w_N,\SSS_n}}_{N}- U^\gamma_{\SSS_n} \Vert_{L^\infty}= 0.
\]
\end{propo}

\begin{proof}
For $G \in \{\AAA_n,\SSS_n\}$, let  $X_{1,n},\ldots,X_{N,n}$ be independent random variables on $G$, with distribution $p_{w_1,G}, \ldots,p_{w_N,G}$
which are invariant under conjugation in $G$.
Let $X^N_n = X_{1,n}\cdots  X_{N,n}$.
For $g\in G$,
\begin{equation}
\label{prob3}
\Pr[X^N_n=g] = |G|^{-1} \sum_{C_1,\ldots,C_N}\biggl(\prod_{i=1}^N \Pr[X_{i,n}\in C_i] \biggr)\sum_{\chi} \frac{\chi(C_1)\cdots  \chi(C_N) \bar\chi(g)}{\chi(1)^{N-1}},
\end{equation}
where each $C_i$ in the first summation ranges over the  conjugacy classes of $G$, and the second summation runs over all
irreducible characters $\chi \in \Irr(G)$. In the second summation, the contribution of the $\chi=1_G$ term is
$1$.

Suppose $G = \SSS_n$. Then, by Lemma \ref{parity2} and \eqref{for-gamma}, the contribution of $\chi = \sgn$ to \eqref{prob3} is
\begin{equation}\label{for-gamma2}
\begin{aligned}
 \sgn(g)|G|^{-1}\sum_{C_1,\ldots,C_N}\biggl(\prod_{i=1}^N \Pr[X_{i,n}\in C_i]\sgn(C_i) \biggr) & =\\
 \sgn(g)|G|^{-1}\biggl(\prod_{i=1}^N \sum_{C_i}\Pr[X_{i,n}\in C_i]\sgn(C_i) \biggr) & = \gamma \cdot \sgn(g)|G|^{-1}.
\end{aligned}
\end{equation}

Our goal is to show that the \emph{error term} in (\ref{prob3}), i.e., the contribution of the characters $\chi$ with $\chi(1)>1$
to the sum, is $o(|G|^{-1})$.  Indeed, in the case
$G = \AAA_n$ we then have for all $g$:
$$|G|\cdot \left|\Pr[X^N_n=g] -U(\{g\})\right| = o(1).$$
Suppose $G = \SSS_n$. Then $U^\gamma(\{g\}) = (1+\gamma\cdot\sgn(g))/|G|$, and so \eqref{for-gamma2} yields
$$|G|\cdot \left|\Pr[X^N_n=g] -U^\gamma(\{g\})\right| = o(1).$$

\smallskip
We choose $\epsilon = 1/4$ in Proposition~\ref{bound1} and assume $n>A(\epsilon)$.
We also assume that $n$ is large enough that $n^{3/4} > (el)^{36}$.
For any conjugacy class $C_i$ in $G$ and any irreducible character $\chi$ of $G$, either
$\Fix(C_i)\le n^{3/4}$, in which case
\begin{equation}
\label{small-fix}
\Pr[X_{i,n}\in C_i] |\chi(C_i)| \le \chi(1)^{11/12}
\end{equation}
by Proposition~\ref{bound1}(i), or $\Fix(C_i) \ge n^{3/4} > (el)^{36}$, in which case
\begin{equation}
\label{power1}
\Pr[X_{i,n}\in C_i] |\chi(C_i)| < n^{\Fix(C_i)/4}\chi(1)^{3/4}
\end{equation}
by Proposition~\ref{bound1}(ii), and
\begin{equation}
\label{power2}
\Pr[X_{i,n}\in C_i] |\chi(C_i)| < \Fix(C_i)^{-\Fix(C_i)/3l^4} \chi(1) < n^{-\Fix(C_i)/4l^4} \chi(1)
\end{equation}
by Proposition~\ref{bound2}.
We combine these inequalities by putting the right hand side of  (\ref{power1})
to the $\frac 1{l^4+2}$ power
and the right hand side of (\ref{power2})  to the $\frac{l^4+1}{l^4+2}$ power and multiplying:
\begin{equation}
\label{big-fix}
\Pr[X_{i,n}\in C_i] |\chi(C_i)| < n^{-\frac{\Fix(C_i)}{4l^4+8}}\chi(1)^{\frac{4l^4+7}{4l^4+8}}
\le n^{-\frac{n^{3/4}}{4l^4+8}}\chi(1)^{\frac{4l^4+7}{4l^4+8}}.
\end{equation}

The error term
$$|G|^{-1}\sum_{C_1,\ldots,C_N}\biggl(\prod_{i=1}^N \Pr[X_{i,n}\in C_i] \biggr)\sum_{\chi(1)>1} \frac{|\chi(C_1)\cdots  \chi(C_N) \bar\chi(g)|}{\chi(1)^{N-1}}$$
can be rewritten as
\begin{equation}
\label{error}
|G|^{-1}\sum_{\chi(1)>1}|\bar\chi(g)|\chi(1)
\sum_{C_1,\ldots,C_N}\biggl(\prod_{i=1}^N \frac{\Pr[X_{i,n}\in C_i] |\chi(C_i)|}{\chi(1)} \biggr).
\end{equation}

If $\Fix(C_i) \ge n^{3/4}$ for at least $25$ different values $i\in [1,N]$, then by (\ref{small-fix}),
$$\prod_{i=1}^N \frac{\Pr[X_{i,n}\in C_i] |\chi(C_i)|}{\chi(1)} \le
\chi(1)^{-25/12} \prod_{i=1}^n \Pr[X_{i,n}\in C_i].$$
Now,
$$\sum_{\chi(1)>1} |\bar\chi(g)|\chi(1) \cdot \chi(1)^{-25/12} \sum_{C_1,\ldots, C_N}\prod_{i=1}^n \Pr[X_{i,n}\in C_i]
= \zeta_G(1/12)-[G:\AAA_n]$$
goes to zero by Theorem 2.6 and Corollary 2.7 of  \cite{LiSh2}.  Therefore, the contribution of
$N$-tuples $(C_1,\ldots,C_N)$ of which at least 25 of the $C_i$ have fixity $\le n^{3/4}$ to the the error term (\ref{error}) is $o(|G|^{-1})$.  That leaves the $N$-tuples for which at least $8l^4+16$ classes $C_i$ have fixity $\ge n^{3/4}$.  For these, the bound (\ref{big-fix}) gives
$$\prod_{i=1}^N \frac{\Pr[X_{i,n}\in C_i] |\chi(C_i)|}{\chi(1)} \le
n^{-2n^{3/4}}\chi(1)^{-2}.$$
The total number of ordered $N$-tuples of conjugacy classes in $G$ is at most $(2p(n))^N \leq e^{cN\sqrt  n}$, where $p(n)$ denotes the partition function.  Thus,
$$\sum_{\chi(1)>1} |\bar\chi(g)|\chi(1)\sum_{C_1,\ldots,C_N} n^{-2n^{3/4}}\chi(1)^{-2} = o(1),$$
which implies the proposition.
\end{proof}

\begin{propo}
\label{waring-classical}
Let $l$ and $N$ be positive integers such that $N>(1.5) \cdot 10^{10} l^2$.
If $w_1, w_2, \ldots , w_N$ is a sequence of non-trivial words of length at most $l$, then
we have
\[
\lim_{|G|\to  \infty} \Vert \underbrace{p_{w_1,G}\ast \cdots  \ast p_{w_N,G}}_{N}- U_{G} \Vert_{L^\infty}= 0
\]
if the limit is taken over finite simple groups $G$ of Lie type of rank $r \geq 7 \cdot 10^8 l^2$.
\end{propo}

\begin{proof}
We follow the method of Proposition~\ref{waring-alt}, and let $W_i$ denote a random variable with values in $G$ and distribution $p_{w_i,G}$.
We write $G$ as the quotient of $H=\Cl(V)$ by its center, where $\F = \F_{q^f}$, $f\le 2$, and  $V = \F^n$.  Let $\tilde W_i$ denote a random variable on $H$ with distribution
$p_{w_i,H}$.  Let $\delta_0 = 0.0011$ and $\delta = 1/7400$.  Note that $\delta (2n+2)^2 < \delta_0 n^2$, since $n \geq r$.
Hence, by \cite[Theorem~1.4]{GLT3}, if $|\CB_G(g)| \le q^{2f\delta n^2}$
then $|\chi(g)| \le \chi(1)^{1-.008}$.

By Lemma~\ref{GtoH},
$$\Pr[|\CB_G(W_i)| \ge q^{2f\delta n^2}] \le \Pr[|\CB_{H}(\tilde W_i)| \ge q^{2f\delta n^2}].$$
By Lemma~\ref{centralizer} (and the discussion prior to \eqref{cond3}), $\Pr[|\CB_{H}(\tilde W_i)| \ge q^{2f\delta n^2}]$
does not exceed the probability that there exists a non-constant polynomial $Q(x)\in \F[x]$ such that
$\dim_{\F}\ker Q(\tilde W_i) \ge \delta n\deg Q$.
As $n\ge r> 7 \cdot 10^8l$,
$n\delta \ge 9 \cdot 10^4 l$, and so
$$k  = \lfloor  n\delta/2l\rfloor \ge n\delta/3l+3.$$
From Proposition~\ref{big-kernel}, it follows that
$$\Pr[|\CB_G(W_i)| \ge q^{2f\delta n^2}] \le q^{-fk((k-1)lD-2.5)}\le q^{-\delta^2fn^2/9l^2}\le |G|^{-\delta^2/9l^2}.$$
Setting $\epsilon = \frac{1}{5\cdot 10^8 l^2} < .008$, for every $i$ and conjugacy class $C_i$, we have that
either $|\chi(C_i)| \le \chi(1)^{1-\epsilon}$ for all $\chi \in \Irr(G)$,
or $\Pr[W_i\in C_i] \le |G|^{-\epsilon}$.

By hypothesis, $G$ has rank $r\ge 11/(8\epsilon)$; in particular $1.1Nr \leq (4/5)Nr^2\epsilon$.
Recall that by \cite{FG}, the number $k(G)$ of conjugacy classes in $G$  is less than $27.2 q^r < q^{r+5} < q^{1.1r}$,
while $|G| \geq q^{hr}/2(r+1)$, where $h\ge r+2$ is the Coxeter number of $G$, so $|G| \ge q^{r^2}$.
As $N\ge 11/\epsilon$ and using the obvious estimate $|\chi(g)\chi(1)| \leq |G|$, we have that
$$\begin{aligned}
|\bar\chi(g)|\chi(1)\sum_{C_1,\ldots,C_N} |G|^{-(9/10)Nr^2\epsilon} & \le q^{1.1Nr}|G|^{1-(9/10)N\epsilon}\\
& \leq q^{1.1Nr-(4/5)Nr^2\epsilon}q^{r^2(1-N\epsilon/10)} \le q^{-0.1r^2}.
\end{aligned}$$
It follows that
$$\sum_{\chi}|\bar\chi(g)|\chi(1)\sum_{C_1,\ldots,C_N} \prod_{i=1}^N \biggl(\frac{\Pr[X_{i,n}\in C_i] |\chi(C_i)|}{\chi(1)} \biggr) \le q^{-0.1r^2},$$
if the inner sum is taken only over $N$-tuples $(C_1,\ldots,C_N)$ for which at least $0.9N$ of the classes satisfy $\Pr[W_i\in C_i] \le |G|^{-\epsilon}$.
Certainly, $q^{-0.1r^2} \to 0$ when $|G| \to \infty$.
On the other hand, by \cite[Theorem~1.1]{LiSh3}, 
$$\lim_{|G|\to \infty} \zeta_G(1)-1 = 0.$$ 
As $N>30/\epsilon$, we obtain that
$$\begin{aligned}\sum_{\chi\neq 1_G}|\bar\chi(g)|\chi(1)\sum_{C_1,\ldots,C_N} \prod_{i=1}^N \biggl(\frac{\Pr[X_{i,n}\in C_i] |\chi(C_i)|}{\chi(1)} \biggr) & \le
    \zeta_G\biggl(\frac{N\epsilon}{10}-2\biggr)-1\\ & \leq \zeta_G(1)-1,\end{aligned}$$
if the inner sum is taken over $N$ tuples $(C_1,\ldots,C_N)$ for which at least $0.1N$ classes satisfy $|\chi(C_i)| \le \chi(1)^{1-\epsilon}$ for all irreducible $\chi$.
\end{proof}

\begin{propo}
\label{waring-r}
Let $r$ and $N$ be positive integers such that $N\geq (2r+1)^2$ and $w_1, w_2, \ldots , w_N$ be any sequence of non-trivial words.  Then
\[
\lim_{|G|\to  \infty} \Vert \underbrace{p_{w_1,G}\ast \cdots  \ast p_{w_N,G}}_{N}- U_{G} \Vert_{L^\infty}= 0
\]
if the limit is taken over finite simple groups $G$ of Lie type of rank $r$.
\end{propo}

\begin{proof}
We follow the method of Proposition~\ref{waring-alt}.  Let $W_i$ denote a random variable with values in $G$ and distribution $p_{w_i,G}$.
By \cite{FG}, the number $k(G)$ of conjugacy classes in $G$ is $< 28q^r$.  (Note $q$ need not be in $\Z$ since $G$ may be of Suzuki or Ree type.)
By classification of root systems, $\dim G \le 2r^2+r$, where by a slight abuse of notation we write $\dim G$ for the dimension of the
simply connected algebraic group associated to $G$.  Thus,
$$\chi(1) \le |G|^{1/2} = O(q^{r^2+r/2}).$$
By Gluck's bound \cite{Gl} for irreducible character values of groups of Lie type, $|\chi(g)| = O(q^{-1/2} \chi(1))$ for all $g\neq 1$.
On the other hand $\Pr[W_i = 1] = \frac{|w^{-1}(1)|}{|G|^d}$.
By \cite[Proposition~3.4]{LS3}, this is $O(q^{-1})$.
Thus, for every conjugacy class $C_i$, either $C_i = \{1\}$, in which case
$\Pr[W_i\in C_i] = O(q^{-1})$, or $C_i\neq \{1\}$, in which case $|\chi(C_i)|/\chi(1) = O(q^{-1/2})$.

Let $S$ denote any subset of $\{1,\ldots,N\}$, and let $C(S)$ denote $N$-tuples $(C_1,\ldots,C_N)$
such that $C_i =\{1\}$ if and only if $i\in S$.  Then,
$$\sum_{(C_1,\ldots,C_N)\in C(S)}\prod_{i=1}^n \biggl(\frac{\Pr[W_i\in C_i] |\chi(C_i)|}{\chi(1)} \biggr) = O(q^{-|S|-\frac{N-|S|}2}) \le O(q^{-N/2}).$$
As $|\Irr(G)| = k(G) < 28q^r$,
\begin{align*}
\sum_{\chi\neq 1_G}|\bar\chi(g)|\chi(1)\sum_S \sum_{(C_1,\ldots,C_N)\in C(S)}\prod_{i=1}^n \biggl(\frac{\Pr[W_i\in C_i] |\chi(C_i)|}{\chi(1)} \biggr)
&= O(q^{2r^2+2r-N/2}) \\ &\le O(q^{-1/2}),
\end{align*}
which proves the proposition, since $q \to \infty$ when $|G| \to \infty$.
\end{proof}

\begin{proof}[Proof of Theorem~\ref{main-inf}]
By the classification of finite simple groups,
it suffices to prove the result for alternating groups, groups of Lie type of rank greater than $7\cdot 10^8 l^2$,
and groups of Lie type of lower rank.
These three cases are covered by Propositions~\ref{waring-alt}, \ref{waring-classical}, and \ref{waring-r} respectively.
\end{proof}

\section{Some applications}

In this section we derive various applications, proving Theorem \ref{applications}.

\begin{propo}\label{parabolic}
There is an absolute constant $0 < \e < 1$ such that the following statement holds for any prime power $q$ and for any positive integer $n$.
Let $G = \Cl(V)$ be a classical group in dimension $n$ (in the sense of \S5), and let $P$ be a maximal subgroup of $G$ of order
$|P| > |G|^{1-\e}$
not containing $[G,G]$. Then there is a classical group $H$ in dimension $m$, with $m < n$ and a normal subgroup $K \lhd P$
with $|K| < [G:P]^3$ such that $P/K \cong H$.
\end{propo}

\begin{proof}
The smallest index of proper subgroups of a simple finite classical group is listed in \cite[Table 5.2.A]{KlL}. It follows for $n \leq 9$ that,
if we take $0 < \e < 1/10$ small enough, then any maximal subgroup $P$ of $G$ of order $|P| > |G|^{1-\e}$ contains $[G,G]$. Hence in the rest of
the proof we may assume $n \geq 10$.
By Theorems 1 and 2 of \cite{K},  there is some $\theta(n) \geq 1$ such that $P$ acts reducibly on $V$
whenever $[G:P] < q^{\theta(n)}$ and $P \not\geq [G,G]$; furthermore $\theta(n) > d(G)/5$ when $n$ is large enough,
where $d(n)$ is the degree of $|[G,G]|$ as a polynomial of $q$.  By taking $0 < \e \leq 1/5$ small enough, we may therefore assume
that $P$ stabilizes a subspace $U \subset V$ of dimension $0 < k < n$. The maximality of $P$ then implies that
$P = \Stab_G(U)$, and furthermore, if $G$ respects a form $\langle\cdot,\cdot\rangle$,
$U$ is either totally singular or non-degenerate with respect to $\langle\cdot,\cdot\rangle$.

First suppose that $\SL_n(\F_q)\le G\le \GL_n(\F_q)$. Then $[G:P] > q^{k(n-k)}$. Replacing the action of $P$ on $U$ by its action on $V/U$ if necessary,
we may assume that $k \geq n/2$, and get a surjection from $P$ onto $H = \GL_k(\F_q)$, with kernel $K$ of order less than
$$q^{k(n-k)+(n-k)^2} = q^{(n-k)n} \leq q^{2k(n-k)} < [G:P]^2.$$

Next suppose that $\SU_n(\F_q)\le G\le \Un_n(\F_q)$. If $U$ is non-degenerate, then replacing it by $U^\perp$ if necessary, we may assume that $k \geq n/2$. The action
of $P$ on $U$ yields a surjection from $P \leq \Un_k(\F_q) \times \Un_{n-k}(\F_q)$ onto $H = \Un_k(\F_q)$,
with kernel $K \leq \Un_{n-k}(q)$ of order less than
$$q^{(n-k)^2+1} \leq q^{2k(n-k)-3} < [G:P].$$
If $U$ is totally singular, then $k \leq n/2$ and
$$q^{2kn-3k^2-4} < [G:P] < q^{2kn-3k^2+1}.$$
By taking $0 < \e \leq 1/5$ small enough, the condition $|P| > |G|^{1-\e}$ implies that $n-2k \geq 6$. Now the action of $P$ on
$U^\perp/U$ yields a surjection from $P$ onto $H = \Un_{n-2k}(q)$ with kernel $K$ of order less than
$$q^{k(2n-3k)+2k^2} \leq q^{3k(2n-3k)-12} < [G:P]^3.$$
The orthogonal and symplectic cases are handled in the same way.
\end{proof}

Recall that the group $\Gamma$ is defined in Theorem \ref{applications} as the group
with generators $x_1, \ldots , x_d$ and a single relator $w = w_1 \cdots w_N$, where $w_i \in F_d$ are pairwise disjoint
non-trivial words of length at most $l$.

\begin{propo}\label{sn-impr}
If $N$ is sufficiently large in terms of $l$ then
there exists $\epsilon > 0$ such that for all positive integers $n$, if $m < n^\epsilon/2$ divides $n$, then
$$|\Hom(\Gamma, \SSS_{n/m} \wr \SSS_m)| = (1+o(1))| \SSS_{n/m} \wr \SSS_m)|^{d-1}.$$
\end{propo}

\begin{proof}
The proof is essentially that of Proposition~\ref{waring-alt}, but it requires three estimates for $H =  \SSS_{n/m} \wr \SSS_m$ analogous to those
for  $\SSS_n$ and $\AAA_n$ used in that proof: an upper bound on the probability that
a random variable with distribution
$w_* U_{H^d}$ takes values with more than $2el^4\sqrt n$ fixed points, an upper bound on values of irreducible characters $\chi$ of $H$ on elements with
$\le 2el^4\sqrt n$ fixed points, and an upper bound on $|\Irr(H)|$.

For the first, we fix $\epsilon < 1/4l$ and apply Proposition~\ref{nearly-trans}.  Indeed, since $n/m \ge  2n^{1-\epsilon}$, the orbit $\cO_{H_S}(u)$ of any
element $u\in [1,n]$ under the pointwise stabilizer $H_S$ of any subset $S\subset [1,n]\setminus\{u\}$
with $|S| \le n^{1-\epsilon}$,  satisfies $|\cO_{H_S}(u)| \ge n^{1-\epsilon}$.  The second is given by Lemma~\ref{wr}.
The third follows from the classification of irreducible characters of $H$ used in the proof of Lemma~\ref{wr}; namely, each such character is
determined by an ordered $m$-tuple of irreducible characters of $\SSS_{n/m}$ together with an irreducible character of $\SSS_m$.  Thus, $|\Irr(H)| \le p(n/m)^m p(m)$. As $m < n^\epsilon$,
we conclude that
$$\log |\Irr(H)| = O(n^{\frac{1+\epsilon}2}).$$
\end{proof}


\begin{lem}
\label{small-max}
Let $d \in \Z_{\geq 1}$ and let $\Delta$ be any $d$-generated group such that
$$\lim_{|G|\to \infty} |G|^{1-d} |\Hom(\Delta,G)| = 1,$$
where $G$ ranges over the finite simple groups.
If $0< \epsilon < 1/(d-2)$, then the probability that
a homomorphism $\varphi\colon \Delta\to G$ chosen uniformly from $\Hom(\Delta,G)$
has the property that $\varphi(\Delta)$ is contained in a maximal subgroup $M$ of $G$ of index greater than $|G|^\epsilon$ goes to $0$ as $|G|\to \infty$.
\end{lem}

\begin{proof}
For any finite simple group $G$, let $Q(G)$ denote the
probability that a random homomorphism from $\Delta$ to $G$ is not an epimorphism. Then
\[
Q(G) \le \sum_{M \stackrel{\max}{<} G} |\Hom(\Delta, M)|/|\Hom(\Delta, G)|.
\]
Since $\Delta$ is $d$-generated, we trivially have $|\Hom(\Delta, M)| \le |M|^d$. Thus
\[
Q(G) \le (1+o(1)) \sum_{M \stackrel{\max}{<} G} |M|^d/|G|^{d-1}.
\]
By  \cite[Theorem~1.1]{LMS},
\begin{equation}\label{zeta-sub}
  \sum_{M \stackrel{\max}{<} G} [G:M]^{-2} \to 0
\end{equation}
as $G$ ranges over the finite simple groups.
If $|M| \le |G|^{1-1/(d-2)}$ then $|M|^d/|G|^{d-1} \le [G:M]^{-2}$. This implies
\[
\sum_{M \stackrel{\max}{<} G, |M| \le |G|^{1-1/(d-2)}} |M|^d/|G|^{d-1} \to 0 \; {\rm as} \; |G| \to \infty.
\]

\end{proof}

\begin{propo}
\label{classical-target}
Let $l$ be a positive integer and $w = w_1\cdots w_N\in F_d$ be a product of pairwise disjoint non-trivial words $w_i$, each of length
at most $l$. If $N$ is sufficiently large in terms of $l$, then there exists $0 < \epsilon < 1$ such that the following statement holds.
For every finite classical group $G$
and every  $P < G$ maximal among
subgroups not containing $[G,G]$ with $|P| > |G|^{1-\epsilon}$, we have
$$|w_P^{-1}(1)| \le  [G:P]^{-2} |G|^{d-1}.$$
\end{propo}

\begin{proof}
By Proposition~\ref{parabolic}, we can find $0 < \epsilon < 1$ such that, given any $P$ as in the theorem,
there exists a classical quotient group $H$ with
\begin{equation}\label{index-ph}
  |P|/|H| < [G:P]^3,
\end{equation}
and so
$$|w_P^{-1}(1)| \le (|P|/|H|)^d \max_{h\in H} |w_H^{-1}(h)| < [G:P]^{3d} \max_{h\in H} |w_H^{-1}(h)|.$$
We claim that the right hand side is $O([G:P]^{3d+3}|H|^{d-1})$.  It suffices to prove that the maximum of $|w_H^{-1}(h)|$
is $O(q^3 |H|^{d-1})$.
We follow the method of proof of Proposition~\ref{waring-alt}.  We start with the inequality
$$|w_H^{-1}(h)| \le |H|^{d-1}\sum_{\chi\in \Irr(H)} \sum_{C_1,\ldots,C_N} \frac{|\chi(C_1)\cdots\chi(C_N)|}{\chi(1)^{N-2}}\prod_{i=1}^N \Pr[W_i\in C_i],$$
where the $W_i$ are independent random variables of $H$ with distribution $p_{w_i,H}$, and the $C_i$ are the conjugacy classes of $H$.
We separate this sum into two pieces according to whether the restriction of $\chi$ to $[H,H]$ has a trivial constituent.
The contribution of the characters whose restriction to $[H,H]$ has a trivial constituent is at most
$$\sum_{\chi \in \Irr(H/[H,H])}\chi(1)^2 = |H/[H,H]| < 2q.$$

Let  $\Irr(H)^*$ denote the set of characters whose restriction to $[H,H]$ has no trivial constituent.
If for some positive constant $\epsilon$ depending only on $l$ we have that for every conjugacy class $C_i$ of $H$
and every irreducible character $\chi$ of $H$, either $\Pr[W_i\in C_i] < |H|^{-\epsilon}$ or $|\chi(C_i)| \le \chi(1)^{1-\epsilon}$,
and, moreover,
\begin{equation}
\label{chi-star}
\lim_{|H|\to \infty} \sum_{\chi\in\Irr(H)^*} \chi(1)^{-3} = 0,
\end{equation}
then we can finish as in Proposition~\ref{waring-classical}.  The dichotomy for $C_i$ follows for general classical groups by the same argument as for
finite simple classical groups since Proposition~\ref{big-kernel} and the character estimate \cite[Theorem~1.3]{GLT3}
hold for classical groups in full generality.

To estimate the sum in \eqref{chi-star}, we choose a function 
$$f\colon \Irr(H)^*\to \Irr([H,H])$$
mapping each $\chi$ to a non-trivial irreducible character of $[H,H]$ which appears as a factor of the restriction of $\chi$ to $[H,H]$.
Thus, $f(\chi)(1) \le \chi(1)$, and $f$ is at most $|H/[H,H]| \le 2q$ to $1$.  Thus, it suffices to prove
$$\lim_{|H|\to \infty} \sum_{1\neq \chi\in\Irr([H,H])} 2q\chi(1)^{-3} \to 0,$$
and since the minimum degree of a non-trivial representation of $[H,H]$ is greater than $q/3$, this follows
from \cite[Theorem~1.1]{LiSh3}.
%
\end{proof}

\begin{proof}[Proof of Theorem \ref{applications}]
In the proof of parts (i)--(iv) we take $N^*(l) = N(l)$ (defined in Theorem \ref{main-inf}).
In the proof of parts (v) and (vi) we have $N^*(l) \ge N(l)$.
Hence in any case we have $N^*(l) > 3$ for all $l$, which implies that $d > 3$.

\smallskip
(i) This follows immediately from Theorem~\ref{main-inf} and Proposition~\ref{flat}.

\smallskip
(ii) As $\Hom(\Gamma,\uG)$ is the fiber of the word map $w_{\uG}$ over the identity, this follows from  (i)
and the fact \cite[Corollaire~6.1.2]{EGA42} that all non-empty fibers of a flat morphism of affine varieties $\uX\to \uY$ have dimension $\dim \uX-\dim \uY$.

\smallskip
(iii) The quasi-finite morphisms (i.e., morphisms with finite fibers)
$\SL_n\times \GL_1\to \GL_n$ and $\GL_n \to \PGL_n \times \GL_1$
give rise to quasi-finite morphisms
\[
\Hom(\Gamma,\SL_n)\times \Hom(\Gamma,\GL_1)\to \Hom(\Gamma,\GL_n)\to \Hom(\Gamma,\PGL_n) \times \Hom(\Gamma,\GL_1).
\]
It follows from part (ii) above that
\[
\dim \Hom(\Gamma,\SL_n) = \dim \Hom(\Gamma, \PGL_n) = (d-1)(n^2-1).
\]
We therefore have the inequalities
\[
\begin{aligned}
(d-1)(n^2-1) + \dim \Hom(\Gamma,\GL_1) & \le \dim \Hom(\Gamma,\GL_n)\\
& \le (d-1)(n^2-1) + \dim \Hom(\Gamma,\GL_1),
\end{aligned}
\]
and so
\[
\dim \Hom(\Gamma,\GL_n) = (d-1)(n^2-1) + \dim \Hom(\Gamma,\GL_1).
\]
Note that $\dim\Hom(\Gamma,\GL_1)$ is $d$ or $d-1$ depending on whether $w$ belongs to $[F_d,F_d]$ or not.
Thus $\dim \Hom(\Gamma,\GL_n) = (d-1)n^2 + a$ where $a=0$ if $w \not\in [F_d,F_d]$ and $a=1$ otherwise.

\smallskip
(iv) Note that for any group $\Gamma$ we have
\[
a_n(\Gamma) = |\Hom_{\trans}(\Gamma,\SSS_n)|/(n-1)!,
\]
where $\Hom_{\trans}(\Gamma, \SSS_n)$ is the set of homomorphisms from $\Gamma$ to $\SSS_n$ with transitive image.
See for instance \cite[1.1.1]{LuSe}.

By Proposition \ref{waring-alt}, we have for $\gamma=\gamma(w)$ (and so $b = 1+\gamma$) that
\[
\lim_{n \to \infty} \Vert p_{w,\SSS_n}- U^\gamma_{\SSS_n} \Vert_{L^\infty}= 0.
\]
It follows that
$P_{w,\SSS_n}(1) \sim b/n!$,
hence
\[
|\Hom(\Gamma, \SSS_n)| \sim b \cdot n!^{d-1}.
\]

Now, the probability $Q_n$ that $\phi \in \Hom(\Gamma, \SSS_n)$ is not in $\Hom_{\trans}(\Gamma, \SSS_n)$
satisfies
\[
Q_n \le \sum_{1 \le k \le n/2} {n \choose k} |\Hom(\Gamma, M_k)|/|\Hom(\Gamma, \SSS_n)|,
\]
where $M_k = \SSS_k \times \SSS_{n-k}$, the stabilizer of $\{ 1, \ldots ,k \}$ in $\SSS_n$.
We have
\begin{align*}
|\Hom(\Gamma, M_k)| & = |\Hom(\Gamma, \SSS_k)| \cdot |\Hom(\Gamma, \SSS_{n-k})| \\
& \le (b+o_k(1))k!^{d-1} \cdot (b+o_n(1))(n-k)!^{d-1}.
\end{align*}
Therefore
\[
Q_n \le \sum_{1 \le k \le n/2} {n \choose k}  \frac{(b+o_k(1))(b+o_n(1))}{b+o_n(1)} \cdot \frac{k!^{d-1}(n-k)!^{d-1}}{n!^{d-1}},
\]
and since $d \ge 3$ we see that
\[
Q_n \le \sum_{1 \le k \le n/2} (b+o_k(1)){n \choose k}^{-(d-2)} = O\Bigl(\sum_{1 \le k \le n/2} {n \choose k}^{-1}\Bigr) \to 0  \; {\rm as} \; n \to \infty.
\]
Therefore, as $n \to \infty$, almost all homomorphisms from $\Gamma$ to $\SSS_n$ have transitive image.
This yields
\[
|\Hom_{\trans}(\Gamma, \SSS_n)| \sim |\Hom(\Gamma, \SSS_n)| \sim b n!^{d-1}.
\]
Dividing both sides by $(n-1)!$ we obtain $a_n(\Gamma) \sim bn \cdot n!^{d-2}$, proving the main assertion
of Theorem \ref{applications}(iv). To prove the second assertion, note that $a_n(F_{d-1}) \sim n \cdot n!^{d-2}$ (see \cite[2.1]{LuSe}), and
this yields $a_n(\Gamma)/a_n(F_{d-1}) \to b$ as $n \to \infty$.

To prove the remaining statements in Theorem \ref{applications}, we use another consequence of Theorem \ref{main-inf} that
$$\lim_{|G| \to \infty}|G|^{1-d}|\Hom(\Gamma,G)| = 1$$
when $G$ runs over the finite simple groups. Thus Lemma \ref{small-max} applies to $\Delta = \Gamma$.

\smallskip
(v) Denote by $\Hom_{\prim}(\Gamma, \SSS_n)$ the set of homomorphisms from $\Gamma$ to $\SSS_n$ with
primitive image. Then we have
\[
m_n(\Gamma) = |\Hom_{\prim}(\Gamma, \SSS_n)|/(n-1)!.
\]
The argument is now similar to the one given above in (iv), except that we also have to take the maximal subgroups $M\cap \AAA_n$ with
$M = \SSS_{n/k} \wr \SSS_k$ into account.  Applying Lemma~\ref{small-max} to $\Delta=\Gamma$,
we need only consider $k$-values which are less than $n^\epsilon$.
By Proposition \ref{sn-impr}, the set of homomorphisms  $\varphi\colon \Gamma\to \AAA_n$ 
for which there exists a partition into $k$ sets of cardinality $n/k$ which $\varphi(\Gamma)$ respects has cardinality less than
$$[\SSS_n:M]|\Hom(\Gamma,M)| = (1+o(1)) n! |M|^{d-2} = O(n^{2-d}|\Hom(\Gamma,\AAA_n)|).$$
As $d>3$, the sum over all $k$ values is $o(|\Hom(\Gamma,\AAA_n)|)$.

\smallskip
(vi)
We apply Lemma~\ref{small-max} to $\Delta=\Gamma$ to bound the probability $Q(G)$ defined as in the proof of the lemma.
First assume that $G$ is of Lie type of bounded rank $r$. Choosing $N$ (hence $d$) large, as we may, and using for instance
Tables 5.2.A and 5.3.A of \cite{KlL},
we have that all the maximal subgroups of $G$ satisfy $|M| < |G|^{1-1/(d-1)}$.
So by Lemma \ref{small-max} we obtain that $Q(G) \to 0$
as $|G| \to \infty$ for such simple groups $G$.

\smallskip
For alternating groups $G = \AAA_n$, by Bochert's theorem (see \cite[3.3B]{DM}), for $n$ sufficiently large,
we need only consider maximal subgroups of the types considered in (iv) and (v), and so we are done in this case.

\smallskip
Now let $G$ be a simple classical group of large rank. Choosing $N$ (hence $d$) large, we may assume that
Proposition \ref{classical-target} holds for a fixed $0 < \epsilon < 1/(d-2)$.
If $\Hom(\Gamma,G)$ fails to be surjective,
its image is contained in a subgroup $P$ of $G$ maximal among all subgroups not containing $[G,G]=G$.
By Lemma \ref{small-max}, the probability of this event, but under the condition that $|P| \leq |G|^{1-\epsilon}$
tends to $0$ when $|G| \to \infty$. So it remains to bound this probability under the condition that $|P| > |G|^{1-\epsilon}$.
By Proposition~\ref{classical-target}, for each such $P$, the probability of this is less than $[G:P]^{-2}$.
Together with \eqref{zeta-sub}, this implies that
$Q(G)\to 0$ as $|G|\to \infty$.
\end{proof}


\end{document}